\documentclass[a4paper,11pt]{amsart}
\author{David Holmes}
\address{Mathematisch Instituut Leiden, 
Niels Bohrweg 1,
2333 CA Leiden,
The Netherlands}
\email{holmesdst@math.leidenuniv.nl}

\usepackage{amsmath, amssymb, amsfonts, url, mathrsfs, newclude, amsthm, setspace, pdfsync, verbatim, hyperref}
\usepackage[all]{xy}

\newcommand{\p}{\mathbb{P}}

\newcommand{\oo}{\mathcal{O}}
\newcommand{\qq}{\mathbb{Q}}
\newcommand{\zz}{\mathbb{Z}}

\newcommand{\R}{\mathbb{R}}
\newcommand{\hh}{\operatorname{h}}

\newcommand{\uv}{_{\nu}}

\newcommand{\defeq}{\stackrel{\text{\tiny def}}{=}}

\DeclareMathOperator{\ord}{ord}

\DeclareMathOperator{\length}{length}

\DeclareMathOperator{\HH}{H}

\newcommand{\Span}[1]{\left<#1\right>}

\newcommand{\cc}{\mathscr{C}}

\newcommand{\intersect}[3][\nu]{\operatorname{\iota_{#1}}\left(#2,#3\right)}

\newcommand{\abs}[1]{\left\lvert#1\right\rvert}

\newcommand{\algcl}{^{\textrm{alg}}}

\newcommand{\on}[1]{\operatorname{#1}}

\theoremstyle{definition}

\newtheorem{df}{Definition}

\newtheorem{remark}[df]{Remark}

\newtheorem{assumption}[df]{Assumption}

\theoremstyle{plain}

\newtheorem{proposition}[df]{Proposition}
\newtheorem{lemma}[df]{Lemma}
\newtheorem{theorem}[df]{Theorem}

\newtheorem{corollary}[df]{Corollary}

\date{\today}

\title[Heights on hyperelliptic Jacobians]{An Arakelov-theoretic approach to na\"ive heights on hyperelliptic Jacobians}

\begin{document}
\newcommand{\done}{\delta_3}
\newcommand{\dtwo}{\delta_1}
\newcommand{\dthree}{\delta_2}
\newcommand{\dfour}{\delta_4}
\newcommand{\deltalambda}{\delta_5}
\newcommand{\dfive}{\delta_6}
\newcommand{\dsix}{\delta_7}
\newcommand{\dseven}{\delta_8}
\newcommand{\deight}{\delta_9}

\maketitle

\newcounter{nootje}
\setcounter{nootje}{1}
\renewcommand\check[1]{[*\thenootje]\marginpar{\tiny\begin{minipage}{20mm}\begin{flushleft}\thenootje : #1\end{flushleft}\end{minipage}}\addtocounter{nootje}{1}}

\renewcommand\check[1]{}

\newcommand{\dd}{{\operatorname{d}}}

\begin{abstract}
We use Arakelov theory to define a height on divisors of degree zero on a hyperelliptic curve over a global field, and show that this height has computably bounded difference from the N\'eron-Tate height of the corresponding point on the Jacobian. We give an algorithm to compute the set of points of bounded height with respect to this new height. This provides an `in principle' solution to the problem of determining the sets of points of bounded N\'eron-Tate heights on the Jacobian. We give a worked example of how to compute the bound over a global function field for several curves, of genera up to 11. 
\end{abstract}

\setcounter{tocdepth}{1}
\tableofcontents

\section{Introduction}
\subsection{Previous explicit computational work on N\'eron-Tate heights}

The N\'eron-Tate height was defined by N\'eron \cite{neron1965quasi}. The problems of computing the height of a given point on the Jacobian of a curve and computing the (finite) sets of points of bounded height on the Jacobian have been studied since the work of Tate in the 1960s, who gave a simpler formula for N\'eron's height. Using this formula, Tate (unpublished), Dem'janenko \cite{MR0227166}, Zimmer \cite{MR0419455}, Silverman \cite{MR1035944} and more recently Cremona, Prickett and Siksek \cite{MR2197860}, Uchida \cite{uchida2008difference} and Bruin \cite{bruin2013} have given increasingly refined algorithms for the case of elliptic curves. Meanwhile, in the direction of increasing genus, Flynn and Smart \cite{flynn_and_smart} gave an algorithm for the above problems for genus 2 curves building on work of Flynn \cite{MR1219694}, which was later modified by Stoll (\cite{stoll1999height} and \cite{stoll2002height}). Stoll has announced an extension to the hyperelliptic genus 3 case \cite{stoll_slides}. 

The technique used by all these authors was to work with a projective embedding either of the Kummer variety, or (in the case of Dem'janenko) of the Jacobian itself. Using equations for the duplication maps, they obtain results on heights using Tate's `telescoping trick'. However, such projective embeddings become extremely hard to compute as the genus grows - for example, the Kummer variety is $\p^1$ for an elliptic curve, is a quartic hypersurface in $\p^3$ for genus 2 and for genus 3 hyperelliptic curves is given by a system of one quadric and 34 quartics in $\p^7$ \cite{mullerthesis}. It appears that to extend to much higher genus using these techniques will be impractical. 

In \cite{holmes2010canonical}, the author used techniques from Arakelov theory to give an algorithm to compute the N\'eron-Tate height of a point on the Jacobian of a hyperelliptic curve, and a similar (though different) algorithm for the same problem was given by M\"uller in \cite{SteffenComputation}. Both gave computational examples in much higher genera (9 and 10 respectively) than had been possible with previous techniques. In this paper, we apply Arakelov theory to the problem of computing the sets of points of bounded height. For practical reasons, we will eventually make certain restrictions on the fields considered and on the shape of the curve, namely we insist that the field either has positive characteristic or is $\mathbb{Q}$, and that there is a rational Weierstrass point at infinity. This is discussed in Remark \ref{rk:inf_ass}. 

\subsection{Relation to classical na\"ive heights}
Let $C$ be a hyperelliptic curve over a global field, with marked Weierstrass point $\infty$ and Jacobian $J$. Let $p = [D-g\cdot \infty]$ be a point on the Jacobian $J$, where $D$ is a suitably chosen divisor on the curve $C$. We will define various intermediate heights, but the final na\"ive height of $p$ (denoted $\on{h}^\dagger(p)$) is given by the height of the polynomial which vanishes at the `$x$-coordinates' of points in $D$ (with multiplicity). This is equal to the `classical' na\"ive height of the image of $p$ under the projective embedding given by a certain linear subspace of $\on{H}^0(J, 2 \vartheta)$, where $\vartheta$ is the theta line bundle, i.e. the line bundle associated to the divisor arising as the image of $C^{g-1}$ under the usual map $C^g \rightarrow J$. As such, it is clear that $\on{h}^\dagger \le \hat{\on{h}} + c$ for some constant $c$; the main result of this paper is to give a practical method to \emph{find} a bound. 



\subsection{Practicality regarding searching for points of bounded height}
To determine the number of points of bounded N\'eron-Tate height on a Jacobian, one usually constructs a `na\"ive' height with bounded difference from the N\'eron-Tate height, and then searches for points of bounded na\"ive height. As such, the two main determinants of the speed of such an algorithm will be the size of the bound on the height differences and the dimension of the region in which one must search for points.


\subsubsection{Number fields}
Let $C$ be a curve of genus $g$ over a number field. The algorithm in this paper requires a search region of dimension $g$.  In this paper we do not give a new algorithm for bounding the local Archimedean height difference (see Section \ref{sec:Arch}), but we can estimate the sizes of the bounds produced by techniques in the literature. Bounds using Merkl's theorem \cite{couveignes2011computational} will be extremely large. Indeed, a Merkl atlas must contain at least $2g+2$ charts (since every Weierstrass point must lie at the centre of a chart), and the form of Merkl's theorem then yields a summand like $1200(2g+2)^2 \approx 4800g^2$ in the difference between the heights. A factor like $g^2$ seems hard to avoid (for example such a factor appears again in Lemma \ref{bounding Phi}), but the coefficient $4800$ is very bad from a practical point of view; since these are differences between logarithmic heights, we obtain a factor like $\exp(4800g^2)$ in the ratio of the exponential heights, making a search for rational points unfeasible in practise. The author's PhD thesis \cite{holmesPhDThesis} contains an alternative algorithm that does not make use of Merkl's theorem (and so \emph{may} yield better bounds) but is much more cumbersome to write down. There is some hope that techniques from numerical analysis may give much sharper bounds, but unfortunately they will not readily give \emph{rigorous} bounds. This is important as the main intended application of these results is to \emph{proving} statements about sets of points of bounded height. If you only need something that almost certainly works in practice, then simply hunting for points of `reasonably large' na\"ive height should be sufficient. 

\subsubsection{Function fields}
In the case of a positive-characteristic global field, the height-difference bounds in this paper become substantially smaller, but still not yet small enough to be useful. In Theorem \ref{thm:eg}, we compute bounds for three curves (of genera 2, 4 and 11) over $\mathbb{F}_p(t)$ of the form $y^2 = x^{2g+1}+t$. The bounds we obtain are very roughly of the size $g^4\log p$. Even in the genus 2 example (where we work over $\mathbb{F}_3$, obtaining a bound of $86 \log 3$), to complete a very na\"ive  search for points would require approximately $p^{300}$ factorisations of univariate polynomials over $\mathbb{F}_3$, which is entirely impractical (though with sieving techniques one could hope to do much better). The algorithm presented in this paper is not optimised, so with further work we hope it will be possible in future to make this method practical in some higher genera. 

\subsubsection{Applications}
If the algorithms in this paper can be made practical, they have applications to the problem of saturation of Mordell-Weil groups (see \cite{siksek1995infinite} or \cite{stoll2002height}), to the computation of integral points on hyperelliptic curves (see \cite{integral_points_on_hyp}), to the use of Manin's algorithm \cite{manin_cyclo}, and for numerically testing cases of the Conjecture of Birch and Swinnerton-Dyer. 

\subsubsection{Some open problems}
\begin{itemize}
\item
improve the bounds produced by this algorithm, to make searching for points practical in some small genera;
\item find a practical way to compute bounds at Archimedean places, and even to find good (small) bounds;
\end{itemize}

\subsection{Other algorithms for heights in arbitrary genus}
\label{sec:projective_embeddings}



It appears that it would be possible to extend the projective-embedding-based approaches mentioned above to give `in principle'  algorithms for bounding the difference between the N\'eron-Tate and na\"ive heights for curves of arbitrary genus.  Mumford \cite{mumford_equations_1} and Zarhin and Manin \cite{zarhin_manin_heights} describe the structure of the equations for abelian varieties embedded in projective space and the corresponding heights and height differences, respectively. To apply these results it is necessary to give an algorithm to construct these projective embeddings for Jacobians for curves of arbitrary genus. Work in this direction includes \cite{van1998equations} and \cite{ReidThesis} in the hyperelliptic case, and \cite{anderson2002edited} in the general case. A bound on the difference between the N\'eron-Tate height and the na\"ive height arising from such an embedding is given by Propositon 9.3 (page 665) in the paper \cite{david2002minorations} of David and Philippon, using an embedding of the Jacobian using $16 \vartheta$. An algorithm for the construction of this embedding has yet to be written down.



\subsection{}
This paper bears some resemblance to the final two chapters of the author's PhD thesis \cite{holmesPhDThesis}. The author would like to thank Samir Siksek for introducing him to the problem, and also Steffen M\"uller and Ariyan Javanpeykar for many helpful discussions, as well as very thorough readings of a draft version. Finally, the author is very grateful to the anonymous referee: firstly for a very rapid and helpful report, which has greatly improved the exposition of the paper, and secondly for some {\tt MAGMA} code which substantially improved the bounds obtained in Section \ref{sec:e.g.}. 


\section{Outline} \label{sec:outline}
Let $K$ be a global field, and $L/K$ a finite extension. Write $M_L$ for a proper set of absolute values of $L$, and $\abs{-}_\nu$ for the valuation at an element $\nu \in M_L$ (see Definition \ref{df:norm normalisations} for our conventions regarding these). We define the (absolute) height of an element $x \in L$ by 
$$\on{h}(x) = \frac{1}{[L:K]} \sum_{\nu \in M_L} \log \max(\abs{x}^{-1}_\nu, 1) $$
and $\on{H}(x) = \exp \on{h}(x)$. This extends to give a well-defined height on the algebraic closure $K\algcl$ of $K$. 

The definition of our first na\"ive height is analogous to this. Let $C/K$ be a hyperelliptic curve. For each absolute value $\nu$ of $K$, we will construct a metric or pseudo-metric $\dd_\nu$ on divisors on $C$ which measures how far apart they are in the $\nu$-adic topology. Given a suitable degree-zero divisor $D$ on $C$ corresponding (up to 2-torsion points) to the point $[D]$ on the Jacobian of $C$, we define the na\"ive height of $[D]$ by 
\begin{equation*}
 \on{h}^n([D]) = \sum_{\nu \in M_K} \log \dd_\nu(D, D')^{-1}
\end{equation*}
where $D'$ is a chosen divisor which is linearly equivalent to $-D$ (up to addition of divisors representing 2-torsion points on the Jacobian). Since the curve $C$ is compact and our metrics continuous, the function $\dd_\nu(D, D')^{-1}$ is bounded below uniformly in $D$, and so we may use $\log(-)$ in place of $\log\left(\max(-,1)\right)$. 

We define these metrics at non-Archimedean absolute values in Definition \ref{df:non}. Theorem \ref{thm:non-Arch} bounds the difference of the distance between two divisors and their local N\'eron pairing at a non-Archimedean absolute value. The hardest aspect of this is allowing for the fact that the model of $C$ obtained by taking the closure inside projective space over the integers of $K$ is not in general a regular scheme, so we must compute precisely how the process of resolving its singularities will affect the intersection pairing. In Definition \ref{df:Arch_metric} we define a pseudo-metric on $C$ at each Archimedean absolute value. Theorem \ref{lem:rho} bounds the difference between this pseudo-metric and the local N\'eron pairing.

We apply Theorem \ref{FH} (due to Faltings and Hriljac) to bound the difference between our height and the N\'eron-Tate height. We then write down two more na\"ive heights, with successively simpler definitions, each time bounding in an elementary fashion the difference from the N\'eron-Tate height. We give a method to compute the number of points of bounded height for the simplest of these na\"ive heights, completing the algorithm. In Theorem \ref{thm:eg} we give a worked example of how to compute these bounds for several curves including a genus 11 curve over $\mathbb{F}_{101}(t)$.

\subsection{Setup and notation}
\label{sec:notation}

\newcommand{\Q}{\mathbb{Q}}
\begin{df}\label{df:hyperelliptic_curve}
We work over a fixed global field $K$ with $2 \in K^\times$ and with fixed algebraic closure $K\algcl$. We fix an integer $g>0$ and a non-zero polynomial $f(X,S) = \sum_{i=0}^{2g+2} f_i X^iS^{2g+2-i} \in K[X,S]$ with exactly $2g + 2$ distinct zeroes in $\mathbb{P}^1(K\algcl)$. We denote by $C$ the curve of genus $g$ over $K$ embedded in weighted projective space $\p(1,1,g+1)$ with coordinates $X$, $S$, $Y$,  defined by the equation $Y^2 = f(X,S)$. We call such a curve a \emph{hyperelliptic curve}. We write $x = X/S$, $y = Y/S^{g+1}$, $s = S/X$ and $y' = Y/X^{g+1}$. We often write $x_p$ for the value of $x$ at $p$, etc. 
\end{df}


\begin{df}
We say that a divisor $D$ on $C$ is \emph{semi-reduced} if it is effective and if there does not exist a prime divisor $p$ of $C$ such that $D \ge p + p^-$ (where $p^-$ denotes the image of $p$ under the hyperelliptic involution). In particular, any Weierstrass point appearing in the support of $D$ has multiplicity 1. If in addition we have $\deg(D) \le g$, then we say $D$ is \emph{reduced}. 
\end{df}

\begin{df}\label{df:norm normalisations} 
For a global field $L$, a \emph{proper set of absolute values for} $L$ is a non-empty multi-set of non-trivial absolute values on $L$ such that the product formula holds. We fix once and for all such a multi-set $M_K$ of absolute values for $K$ such that every Archimedean absolute value $\nu$ comes from a embedding of $K$ into $\mathbb{C}$ with the standard absolute value. Given a finite extension $L/K$, we fix a proper multi-set of absolute values $M_L$ for $L$ by requiring that for all absolute values $\nu \in M_L$, the restriction of $\nu$ to $K$ lies in $M_K$. We denote by $M_L^0$ the sub-multi-set of non-Archimedean absolute values and $M_L^\infty$ the sub-multi-set of Archimedean absolute values. 
\end{df}

\newcommand{\base}{B}
\begin{df}
Given a global field $L$, we define the curve $\base_L$ to be the unique normal integral scheme of dimension 1 with field of rational functions $L$ and such that $\base_L$ is proper over $\on{Spec} \zz$. For example, if $L$ is a number field then $\base_L$ is the spectrum of the ring of integers of $L$. 


\end{df}


\section{Non-Archimedean results}



\subsection{Defining metrics}
\label{sec:metrics}

\newcommand{\rrr}{\mathbb{R}}
\renewcommand{\aa}{\mathbb{A}}

\begin{df}\label{df:non}
For each absolute value $\nu \in M_K$, we fix $(K_\nu\algcl,\abs{-}_{\nu})$ to be an algebraic closure of the completion $K_\nu$ together with the absolute value which restricts to $\nu$ on $K \subset K_\nu\algcl$. For non-Archimedean absolute values $\nu$ we define 
\begin{equation*}
\dd_\nu: C(K_\nu\algcl) \times C(K_\nu\algcl) \rightarrow \mathbb{R}_{\ge 0}
\end{equation*}
by
\begin{equation*}
\begin{split}
&\dd_\nu((X_p:S_p:Y_p),(X_q:S_q:Y_q)) \\
& = \left\{ \begin{array}{ll}
	\max\left( \abs{x_p - x_q}_{\nu}, \abs{y_p^{g+1} - y_q^{g+1}}_{\nu}\right)  & \text{if } \abs{X_p}_{\nu} \le \abs{S_p}_{\nu} \text{ and } \abs{X_q}_{\nu} \le \abs{S_q}_{\nu} \\
	\max\left( \abs{s_p - s_q}_{\nu}, \abs{{y'_p}^{g+1} - {y'_q}^{g+1}}_{\nu}\right)  & \text{if } \abs{X_p}_{\nu} \ge \abs{S_p}_{\nu} \text{ and } \abs{X_q}_{\nu} \ge \abs{S_q}_{\nu} \\
	1 & \text{otherwise}\\
\end{array} \right.\\
\end{split}
\end{equation*}
(here as always $x_p = X_p/S_p $ etc). 
\end{df}


\begin{proposition}
\label{prop:is_metric}
For each $\nu \in M^0_K$, $\dd = \dd_\nu$ is a metric on $C(K_\nu\algcl)$. Moreover, for each such $\nu$, we have $\dd_\nu(p,q) \le 1$ for all $p$ and $q$. 
\end{proposition}
\begin{proof} We omit the subscripts $\nu$ from the absolute values. We begin by observing that if $(X:S:Y) \in C(K_\nu\algcl)$ then
\begin{equation*} 
\abs{X} \le \abs{S} \implies \abs{Y} \le  \abs{S}^{g+1} \; \text{and} \; \abs{X} > \abs{S} \implies \abs{Y} \le  \abs{X}^{g+1}. 
\end{equation*}
 Combining this with the fact that $\abs{-}$ is non-Archimedean, we see for all $p$, $q \in C(K_\nu\algcl)$ that $\dd(p,q) \le  1$. 

For showing that $\dd$ is a metric, only the triangle inequality is non-obvious. Let $p = (X_p,S_p,Y_p)$, $q = (X_q,S_q,Y_q)$ and $r = (X_r,S_r,Y_r)$. Suppose firstly that $\abs{X_p} \le \abs{S_p}$, $\abs{X_q} \le \abs{S_q}$ and $\abs{X_r} \le \abs{S_r}$. Then
\begin{equation*}
\begin{split}
& \dd(p,q) + \dd(q,r) \\
& = \max\left(\abs{x_p - x_q},\abs{y_p^{g+1}- y_q^{g+1}}\right) + \max\left(\abs{x_q - x_r},\abs{{y_q^{g+1}} - {y_r^{g+1}}}\right)\\
& \ge \max\left(\abs{{x_p} - {x_q}} + \abs{{x_q} - {x_r}} ,\abs{{y_p^{g+1}}- {y_q^{g+1}}} + \abs{{y_q^{g+1}} - {y_r^{g+1}}}\right) \\
& \ge \dd(p,r).\\
\end{split}
\end{equation*}
The other cases are similar. 
\end{proof}
%

\subsection{A simple formula for the distance function in a special case}

Here we give a simple bound on the logarithm of the distance between two points $p$ and $w$ on $C$ where $w$ is a Weierstrass point. This will be needed in Section \ref{sec:refined_naive}. 

\begin{df}\label{def_lambda}
We write $W$ for the set of Weierstrass point of $C$ (over $K\algcl$). We assume that $C$ has no Weierstrass point with $X$-coordinate zero (cf. Assumption \ref{assumptions on C}). Let $\nu \in M^0_K$. We define $\lambda_\nu$ to be the smallest real number $\ge 1$ such that the following conditions hold.
\begin{itemize}
\item For all Weierstrass points $w \in W$ with  $w \neq \infty$, we have $1/\lambda_\nu \le \abs{x_w}_\nu \le \lambda_\nu$.
\item For all pairs of Weierstrass points  $w$, $w' \in W\setminus\{\infty\}$ with $w \neq w'$ we have $1/\lambda_\nu \le \abs{x_w - x_{w'}}_\nu \le \lambda_\nu.$
\item We have $1/\lambda_\nu \le \abs{f_{2g+1}}_\nu \le \lambda_\nu$, where $f_{2g+1}$ is the leading coefficient of the defining polynomial $f$ of the curve $C$. 
\end{itemize}
\end{df}
Note that $\lambda_\nu = 1$ for all but finitely many $\nu$.

\begin{lemma}
\label{lem:GCD_non_Arch}Let $L/K$ be a finite extension, and let $p$, $w \in C(L)$ with $p \neq w$ be such that $s_p \ne 0$ and $w$ is a Weierstrass point with $s_w \ne 0$. Let $\nu$ be a non-Archimedean absolute value of $L$ extending an absolute value $\nu'$ of $K$. We have
\begin{equation*}
 -\log(\dd_\nu(p,w)) \le \frac{1}{2}\log^+\abs{x_p-x_w}^{-1}_\nu  + (2g+3/2)\log\lambda_{\nu'}. 
\end{equation*}
\end{lemma}
\begin{proof}
%
The formula we must show is equivalent to (at this point we drop the subscripts $\nu$ and $\nu'$)
\begin{equation}
\label{eqn:aim1}
 \dd(p,w)^2 \ge  \min(\abs{x_p - x_w}, 1)/\lambda^{4g+3}. 
\end{equation}
The proof of this inequality falls into a number of cases depending on the valuations of $x_p$, $x_w$ etc. We will only give the details of the case 
\begin{equation*}
1 < \abs{x_w}, \;\;\;\;\; 1< \abs{x_p} \le \lambda. 
\end{equation*}
In this case, we have
\begin{equation*}
\begin{split}
 \dd(p,w) ^2 & =\abs{x_p - x_w} \max\left(\frac{\abs{x_p - x_w}}{\abs{x_p}^2\abs{x_w}^2}, \frac{ \abs{f_{2g+1}}\prod_{w' \in W \setminus\{w, \infty\}}\abs{x_p - x_{w'} }}{  \abs{x_p}^{2g+2}  }     \right)\\
& \ge \frac{ \abs{x_p - x_w}}{\lambda^{2g+2}} \max\left(\abs{x_p - x_w}, \abs{f_{2g+1}}\prod_{w' \in W \setminus\{w, \infty\}}\abs{x_p - x_{w'} }\right). \\
\end{split}
 \end{equation*}
 Now suppose that $\abs{x_p - x_w} < \lambda$ and 
 \begin{equation*}
 \abs{f_{2g+1}}\prod_{w' \in W \setminus\{w, \infty\}}\abs{x_p - x_{w'} } < 1/\lambda^{2g+1}. 
 \end{equation*}
 Then there exists $w_0 \in W \setminus \{w, \infty\}$ such that $\abs{x_{w_0} - x_p} < 1/\lambda$, so by the strong triangle inequality we have
 \begin{equation*}
 \abs{x_w - x_{w_0}} \le \max(\abs{x_w - x_p}, \abs{x_p - x_{w_0}}) < 1/\lambda, 
 \end{equation*}
 a contradiction. Hence 
 \begin{equation*}
 \max\left(\abs{x_p - x_w}, \abs{f_{2g+1}}\prod_{w' \in W \setminus\{w, \infty\}}\abs{x_p - x_{w'} }\right) \ge 1/\lambda^{2g+1}, 
 \end{equation*}
 and Equation \eqref{eqn:aim1} follows.  
 \end{proof}

\subsection{Local N\'eron pairings in the non-Archimedean case}\label{sec:LNP}
\newcommand{\NLP}[1]{[ #1 ]}
We summarise the construction of the local N\'eron pairing at a non-Archimedean place from \cite[IV, \S1]{lang1988introduction}, where more details can be found. This pairing will play a crucial role in allowing us to compare our `distance' function $\dd_\nu$ to the local height pairing at $\nu$.

 Given an absolute value $\nu$ of $K$, we write $\on{Div}^0(C_{K_\nu})$ for the group of degree-zero divisors on the base change of $C$ to the completion of $K$ at $\nu$. The \emph{local N\'eron pairing} at $\nu$ is a biadditive map
\begin{equation*}
\NLP{-,-}_\nu: \left\{(D,E) \in \on{Div}^0(C_{K_\nu}) \times \on{Div}^0(C_{K_\nu}) | \on{supp}(D) \cap \on{supp}(E) = \emptyset\right\} \rightarrow \mathbb{R}. 
\end{equation*}
Its definition depends on whether $\nu$ is an Archimedean or non-Archimedean absolute value; the definition in the Archimedean case will be given in Section \ref{sec:NLP_arch}. 

Let $\nu$ be a non-Archimedean absolute value. Write $\oo_{K_\nu}$ for the ring of integers of the completion $K_\nu$. Let $\cc = \cc_{\oo_{K_\nu}}$ be a proper, flat, regular model of $C$ over $\oo_{K_v}$. We write $\iota\uv$ for the (rational-valued) intersection pairing between divisors over $\nu$ (as defined in \cite[IV, \S1, page 72]{lang1988introduction}). Let $D$ and $E$ be elements of $\on{Div}^0(C_{K_\nu})$ with disjoint support. We extend $D$ and $E$ to horizontal divisors $\overline{D}$ and $\overline{E}$ on $\cc$. Write $\mathbb{Q}\on{FDiv}(C_{K_\nu})$ for the group of $\mathbb{Q}$-divisors on $\cc$ supported on the special fibre $\cc_\nu$. We define a map (cf. \cite[III, \S 3]{lang1988introduction})
\begin{equation*}
\Phi:  \on{Div}^0(C_{K_\nu}) \rightarrow \frac{\mathbb{Q}\on{FDiv}(C_{K_\nu})}{\mathbb{Q}(\cc_\nu)}
\end{equation*}
by requiring that for all fibral divisors $Y \in \on{FDiv}(C_{K_\nu})$, we have
\begin{equation*}
\intersect{Y}{\overline{D} + \Phi(D)} = 0. 
\end{equation*}
Then define the local N\'eron pairing by
\begin{equation*}
\NLP{D,E}_\nu = \log (\#\kappa) \intersect{\overline{E}}{\overline{D} + \Phi(D)}, 
\end{equation*}
where $\kappa$ is the residue field at $\nu$. 


\begin{proposition}
The local N\'eron pairing at a non-Archimedean absolute value $\nu$ is independent of the choice of regular model $\cc_{\oo_{K_v}}$. 
\end{proposition}
\begin{proof}
Combine Theorem 5.1 and Theorem 5.2 of \cite[III]{lang1988introduction}. 
\end{proof}

\subsection{Comparison of the metric and the N\'eron pairing}
The main aim of this section is to prove the following result: 
\begin{theorem}\label{thm:non-Arch}
Given a non-Archimedean absolute value $\nu \in M_K^0$, there exists an explicitly computable constant $\mathscr{B}_\nu$ with the following property: 

Let $D = D_1 - D_2$ and $E = E_1 - E_2$ be differences of reduced divisors on $C$ with no common points in their supports, and assume that $D$ and $E$ both have degree zero. 
Let $L$ denote the minimal field extension of $K_\nu$ such that $D$ and $E$ are pointwise rational over $L$, and over $L$ write $D = \sum_i d_ip_i$, $E = \sum_j e_jq_j$, with $d_i$, $e_j \in \mathbb{Z}$ and $p_i$, $q_j \in C(L)$. Recall from Section \ref{sec:LNP} that $\NLP{D,E}_\nu$ denotes the local N\'eron pairing of $D$ and $E$ at $\nu$. Then
\begin{equation*}
\abs{ \NLP{D,E}_\nu - \sum_{i,j}d_ie_j\log\left(\frac{1}{\dd_\nu(p_i,q_j)} \right) } \le \mathscr{B}_\nu.
\end{equation*}
 Moreover, if $C$ has a smooth proper model over $\nu$, then we may take $\mathscr{B}_\nu = 0$. 
\end{theorem}
\noindent The proof of this result is postponed to the end of this section. 

For the remainder of this section we fix a non-Archimedean absolute value $\nu \in M^0_K$. Write $\cc_1$ for the Zariski closure of $C: Y^2 = F(X,S)$ in $\p_{\oo_{K_\nu}}(1,1,g+1)$. A result of Hironaka, contained in his appendix to \cite{MR775681} (pages 102 and 105) gives us an algorithm to resolve the singularities of $\cc_1$ by a sequence of blowups at closed points and along smooth curves (the latter replacing the normalisations used in Lipman's algorithm \cite{Lipman}); we observe that $\cc_1$ may locally be embedded in $\p^2_{\oo_{K_\nu}}$, and so Hironaka's result can be applied. We fix once and for all a choice of resolution $\cc$ of $\cc_1$ using this algorithm of Hironaka - thus we fix both the model $\cc$ and the sequence of blowups at smooth centres used to obtain it.

We begin by bounding the function $\Phi$. Let $F$ denote the free abelian group generated by prime divisors supported on the special fibre of $\cc$ over $\nu$, and let $V$ denote the finite-dimensional $\qq$-vector space obtained by tensoring $F$ over $\zz$ with $\qq$. Let $M: V \times V \rightarrow \qq$ be the map induced by tensoring the restriction of the intersection pairing on $\cc$ to its special fibre with $\qq$. Then $V$ has a canonical basis of fibral prime divisors, so we may confuse $M$ with its matrix in this basis. Call the basis vectors $Y_1 \ldots Y_n$; we use the same labels for the corresponding fibral prime divisors. 

\begin{lemma}\label{bounding Phi}
Let $M^+$ denote the Moore-Penrose pseudo-inverse (see \cite{moore1920}, \cite{MR0069793}) of $M$, let $m_-$ denote the infimum of the entries of $M^+$ and $m_+$ their supremum. 
Let $D = D^+ - D^-$ and $E = E^+ - E^-$ be differences of reduced divisors on $C$ with no common points in their supports, and assume that $D$ and $E$ both have degree zero. Then
\begin{equation*}
\abs{\intersect{\Phi(D)}{\overline{E}}} \le 2g^2(m_+ - m_-). 
\end{equation*}

\end{lemma}
\begin{proof}
For each $1 \le i \le n$, set
\begin{equation*}
\begin{split}
d_i^+ = \intersect{\overline{D}^+}{Y_i}, \; & \; d_i^- = \intersect{\overline{D}^-}{Y_i}, \\
e_i^+ = \intersect{\overline{E}^+}{Y_i}, \; & \; e_i^- = \intersect{\overline{E}^-}{Y_i}, 
\end{split}
\end{equation*}
and note that all $d_i^{\pm}$ and $e_i^{\pm}$ are non-negative. Then for each $i$ set 
\begin{equation*}
d_i = d^+_i - d^-_i, \;\; e_i = e^+_i - e^-_i, 
\end{equation*}
and define vectors in $V$ by 
\begin{equation*}
\begin{split}
& d = (d_i)_i, \;  d^+ = (d_i^+)_i, \; d^- = (d_i^-)_i, \\ & e = (e_i)_i, \;  e^+ = (e_i^+)_i, \; e^- = (e_i^-)_i. 
\end{split}
\end{equation*}

 Now by definition of $\Phi$ we have that for all vectors $v \in V$:
\begin{equation*}
v \cdot d^T + v \cdot M \cdot \Phi(D)^T = 0, 
\end{equation*}
and hence that 
\begin{equation*}
d^T = -M\cdot \Phi(D)^T.
\end{equation*}

Recall that if for any matrix $A$ the linear system $Ax = b$ has any solutions, then a solution is given by $x = A^+b$ where $A^+$ is the Moore-Penrose pseudo-inverse of $A$. As such, we can take $\Phi(D)$ to be $-d \cdot \left(M^+\right)^T $, and so we find
\begin{equation*}
\intersect{\Phi(D)}{\overline{E}} = - d \cdot \left(M^+\right)^T \cdot e^T. 
\end{equation*}
Expanding out, we find
\begin{equation*}
\begin{split}
\intersect{\Phi(D)}{\overline{E}} & = - d^+ \cdot \left(M^+\right)^T \cdot (e^+)^T + d^+ \cdot \left(M^+\right)^T \cdot (e^-)^T \\
& \;\;\;\;\;\;\;\; + d^- \cdot \left(M^+\right)^T \cdot (e^+)^T - d^- \cdot \left(M^+\right)^T \cdot (e^-)^T.  
\end{split}
\end{equation*}
We will bound each of these four terms. 

Write $\pi$ for a uniformiser in $\oo_K$ at $\nu$ (so $\nu(\pi) = 1$). Write the divisor of $\pi$ on $\cc$ as $\on{div}(\pi) = \sum_i a_iY_i$, where the $a_i$ are integers greater than 0. Then
\begin{equation*}
\sum_i a_id^+_i = \intersect{\overline{D}^+}{\on{div}(\pi)} = \on{deg} D^+ \le g, 
\end{equation*}
(and similarly for $D^-$ and $E^{\pm}$), the second equality holding by \cite[II, Proposition 2.5]{lang1988introduction}. From this, we see that each $d_i^+ \ge 0$ and $\sum_i d_i^+ \le g$ (and similarly for $d_i^-$ and $e_i^{\pm}$). Hence we find that
\begin{equation*}
\begin{split}
&-g^2m^+\le -d^+(M^+)^T(e^+)^T \le -g^2m^-, \\
&g^2m^-\le d^+(M^+)^T(e^-)^T \le g^2m^+, \\
&g^2m^-\le d^-(M^+)^T(e^+)^T \le g^2m^+, \\
&-g^2m^+\le -d^-(M^+)^T(e^-)^T \le -g^2m^-, \\
\end{split}
\end{equation*}
from which the result follows. 
\end{proof}


\label{sec:Non-Archimedean II: local comparison of metrics and intersection pairings}

We have a chosen resolution $\cc = \cc_{K_\nu}$ (by blowups at smooth centres) of the singularities of the closure $\cc_1$ of $C$ in weighted projective space over $\oo_{K_\nu}$.  
Let $b_\nu$ denote the longest length of a chain of blowups at smooth centres involved in obtaining this resolution (one blowup is considered to follow another if the centre of one blowup is contained in the exceptional locus of the previous one). Note that $b_\nu = 0$ if $\cc_1$ is regular. 

For the remainder of this section, let $D$ and $E$ be effective divisors on $C$ with disjoint support, of degrees $d$ and $e$ respectively. Let $L_\nu/K_\nu$ be the minimal finite extension (of degree $m$ with residue field $l$) such that $D$ and $E$ are both pointwise rational over $L_\nu$. Write $D = \sum_{i=1}^d p_i$ and $E = \sum_{i=1}^e q_i$, and write $\overline{D}$ and $\overline{E}$ for the Zariski closures of $D$ and $E$ respectively on the regular model $\cc_{K_\nu}$ over $\oo_{K_\nu}$ (more precisely, take closures of the prime divisors in the supports of $D$ and $E$, then define $\overline{D}$ and $\overline{E}$ to be appropriate linear combinations of these new prime divisors). Write $\omega$ for the maximal ideal of $\oo_{L_\nu}$. 

\begin{proposition}
\label{prop:finite_comp}
We have
\begin{equation*}
-\log(\#\kappa(\nu)) b_\nu de \le \log(\#\kappa(\nu))\intersect{\overline{D}}{\overline{E}} -\log\left(\frac{1}{\prod_{i,j}\dd(p_i,q_j)}\right) \le 0, 
\end{equation*}
where $\kappa(\nu)$ is the residue field at $\nu$. 
\end{proposition}
\noindent
The proof of Proposition \ref{prop:finite_comp} may be found after Lemma \ref{lema:comp_B}. To avoid an excess of notation, we will from now on drop the subscript $\nu$ from the fields and models we are considering, since we will exclusively be working locally at $\nu$ and places dividing it for the remainder of this section. 

\begin{lemma}
\label{lema:comp_A}
Let $p$, $q \in C(L)$ with $p \neq q$. Write 
\begin{equation*}
I_{p,q} \defeq \sum_{\Omega | \omega} \log(\#\kappa(\Omega))\length_{\oo_L}\left(\frac{\oo_{\cc_1 \times_{\oo_K} \oo_L, \Omega}}{I_p + I_q}\right),
\end{equation*}
where the sum is over closed points $\Omega$ (with residue field $\kappa(\Omega)$) of $\cc_1\times_{\oo_K} \oo_L$ lying over $\omega$, and $I_p$ and $I_q$ are defining ideal sheaves for the closures $\overline{p}$ and $\overline{q}$ in $\cc_1\times_{\oo_K} \oo_L$ of the images of $p$ and $q$ in $C \times_K L$. 
Then
\begin{equation*}
I_{p,q} = m\log\left(\frac{1}{\dd(p,q)}\right)
\end{equation*}
(recall that $m = [L:K]$). 
\end{lemma}
\begin{proof}

Write $p = (X_p:S_p:Y_p)$, $q = (X_q:S_q:Y_q)$ with $X_p$, $S_p$, $X_q$, $S_q \in \oo_L$. If $\abs{X_p} < \abs{S_p}$ and $\abs{X_q} > \abs{S_q}$ or vice versa, then $\overline{p}$ and $\overline{q}$ do not meet on the special fibre so $I_{p,q} = 0$, and by definition we see that $\dd(p,q) = 1$. 

 Otherwise, possibly after changing coordinates, we may assume that $p$ and $q$ are of the form $(x_p:1:y_p)$ and $(x_q:1:y_q)$ respectively, for $x_p$, $y_p$, $x_q$, $y_q \in \oo_L$. We may moreover assume that $\overline{p}$ and $\overline{q}$ meet on the special fibre; let $\Omega$ be the closed point where $\overline{p}$ and $\overline{q}$ meet. 
After multiplying the defining equation $F$ of $C$ on the coordinate chart containing $p$ and $q$ by a power of a uniformiser at $\nu$, we may asume $F$ is integral at $\nu$ and is irreducible. We have
\begin{equation*}
\begin{split}
\frac{\oo_{\cc_1 \times_{\oo_K} \oo_L, \Omega}}{I_p + I_q} & \cong \frac{\oo_L[x,y]_{(x,y)}}{(F,x-x_p,y-y_p,x-x_q,y-y_q)}\\
 & \cong \frac{\oo_L}{(x_p-x_q,y_p-y_q)},\\
\end{split}
\end{equation*}
so 
\begin{equation*}
\length_{\oo_L}\left(\frac{\oo_{\cc_1 \times_{\oo_K} \oo_L, \Omega}}{I_p + I_q}\right) = \min\left(\ord_\omega(x_p-x_q), \ord_\omega(y_p-y_q)\right). 
\end{equation*}

Now given $a \in L$, we find
\begin{equation*}
 \log(\#l)\ord_\omega (a) = -m\log\abs{a}, 
\end{equation*}
so
\begin{equation*}
\length_{\oo_L}\left(\frac{\oo_{\cc_1 \times_{\oo_K} \oo_L, \Omega}}{I_p + I_q}\right) = m\frac{\min\left(-\log\abs{x_p-x_q}, -\log\abs{y_p-y_q}\right)}{\log(\#l)}, 
\end{equation*}
and hence
\begin{equation*}
 I_{p,q} = m\min\left(-\log\abs{x_p-x_q}, -\log\abs{y_p-y_q}\right). 
\end{equation*}
Moreover, 
\begin{equation*}
\log(1/\dd(p,q)) = \min\left(-\log\abs{x_p-x_q} , -\log\abs{y_p-y_q}\right), 
\end{equation*}
so we are done. 
\end{proof}
 
\begin{lemma}
\label{lema:comp_C}
Recalling that over $L$ we can write $D = \sum_{i=1}^d p_i$ and $E = \sum_{i=1}^e q_i$, we define $\oo_{\omega_{i,j}}$ to be the local ring at the closed point of $\cc_1 \times_{\oo_K} \oo_L$ where $p_i$ meets $q_j$ if such exists, and the zero ring otherwise. Letting $\mathcal{I}_D$ and $\mathcal{I}_E$ denote the ideal sheaves of the closures of $D$ and $E$ respectively on $\cc_1$, we have
\begin{equation*}
\sum_{i,j} \length_{\oo_L} \left(\frac{\oo_{\omega_{i,j}}}{I_{p_i} + I_{q_i}}\right) = \length_{\oo_L}\left(\frac{\oo_{\cc_1} \otimes_{\oo_K} \oo_L}{(\mathcal{I}_D + \mathcal{I}_E) \otimes_{\oo_K} \oo_L}\right). 
\end{equation*}
The analogous statement on $\cc$ also holds. 
\end{lemma}
\begin{proof}
We may decompose $\mathcal{I}_D$ and $\mathcal{I}_E$ into iterated extensions of the sheaves $I_{p_i}$ and $I_{q_i}$, whereupon the result follows from additivity of lengths in exact sequences. 
\end{proof}

\newcommand{\ramdeg}{\operatorname{ram.deg}}

\begin{lemma}
\label{lemma:extra_extra}
Let $M$ be a finite length $\oo_K$-module. Then
\begin{equation*}
\length_{\oo_K}(M) \cdot \ramdeg(L/K) = \length_{\oo_L}(M \otimes_{\oo_K} \oo_L). 
\end{equation*}
\end{lemma}
\begin{proof}
Let $M = M_0 \subset M_1 \subset \cdots \subset M_l = 0 $ be a composition series for $M$, so each $M_i/M_{i+1}$ is simple. Since $\oo_K$ is local, we have by \cite[p12]{matsumura_commutative} that
\begin{equation*}
M_i/M_{i+1} \cong \oo_K/\mathfrak{m}_K. 
\end{equation*}
By additivity of lengths, it suffices to show
\begin{equation*}
\length_{\oo_L}\left(\frac{\oo_K}{\mathfrak{m}_K} \otimes_{\oo_K} \oo_L \right) = \ramdeg(L/K), 
\end{equation*}
but this is clear since $\mathfrak{m}_K \cdot \oo_L = \mathfrak{m}_L^{\ramdeg(/K)}$. 
\end{proof}

\begin{lemma}
\label{lema:comp_D}
Let $\mathcal{I}_D$ and $\mathcal{I}_E$ denote the ideal sheaves on $\cc_1$ corresponding to the closures of the divisors $D$ and $E$ respectively. We have:
 \begin{equation*}
\length_{\oo_K}\left(\frac{\oo_{\cc_1}}{\mathcal{I}_D + \mathcal{I}_E}\right) \cdot \ramdeg{L/K} = \length_{\oo_L}\left(\frac{\oo_{\cc_1} \otimes_{\oo_K} \oo_L}{(\mathcal{I}_D + \mathcal{I}_E) \otimes_{\oo_K} \oo_L}\right).
\end{equation*}
The analogous statement on $\cc$ also holds. 
\end{lemma}
\begin{proof}
Setting $M = \frac{\oo_{\cc_1}}{\mathcal{I}_D + \mathcal{I}_E}$, we have that $M$ is a finite-length $\oo_K$-module, and 
\begin{equation*}
M \times_{\oo_K}\oo_L = \frac{\oo_{\cc_1} \otimes_{\oo_K} \oo_L}{(\mathcal{I}_D + \mathcal{I}_E) \otimes_{\oo_K} \oo_L}. 
\end{equation*}
We are done by Lemma \ref{lemma:extra_extra}. 
\end{proof}

\newcommand{\mm}{\mathfrak{m}}
\newcommand{\olp}{\overline{p}}
\newcommand{\olq}{{\overline{q}}}

\begin{lemma}
\label{lema:comp_B}
Let $\phi: \cc_3 \rightarrow \cc_2$ be one of the blowups involved in obtaining $\cc$ from $\cc_1$. Let $p$, $q \in C(L)$ with $p\neq q$. Then 
\begin{equation*}
 0 \le \length_{\oo_L}\left(\frac{\oo_{\cc_2 \times \oo_L}}{I_{\olp} + I_{\olq}}\right) - \length_{\oo_L}\left(\frac{\oo_{\cc_3 \times \oo_L}}{I_{\olp} + I_{\olq}}\right) \le \ramdeg(L/K). 
\end{equation*}
\end{lemma}
\begin{proof}In this proof, we will omit the subscripts  `$\oo_L$' from the lengths, since all lengths will be taken as $\oo_L$-modules. 
 If $\overline{p}$ does not meet $\overline{q}$ on $\cc_2 \times \oo_L$ then both the lengths are zero, so we are done. Otherwise, let $\Omega$ be the closed point on $\cc_2 \times \oo_L$ where $\olp$ meets $\olq$, and let $\alpha$ be the closed point of $\cc_2$ such that $\Omega$ lies over $\alpha$. 

Let $u$, $v$ be local coordinates on the (three-dimensional) ambient space to $\cc_2$ at $\alpha$, and let $R$ denote the completion at $(u,v)$ of the \'etale local ring of the ambient space to $\cc_2$ at $\alpha$. Let $B \subset R$ be the centre of the localisation of $\phi$ at $\alpha$. We have
\begin{equation*}
 R \cong \tilde{\oo}_K[[u,v]]_{(u,v,a)}
\end{equation*}
where $\tilde{\oo}_K$ is the completion of $\oo_K$ and $a$ is a uniformiser in $\tilde{\oo}_K$, and that 
\begin{equation*}
 B = (u,v,a) \;\;\; \text{ or } \;\;\; B = (u,a), 
\end{equation*}
depending on whether we are blowing up a point or a smooth fibral curve. 

Blowups commute with flat base change, and the strict transform of a closed subscheme under a blowup is the corresponding blowup of that closed subscheme (see \cite[Corollary 8.1.17]{liu2006algebraic}), so we can be relaxed with our notation. We may write 
\begin{equation*}
 \begin{split}
p = (u- a u_p, v-a v_p) \;\;\;\;\; q = (u-a u_q, v- a v_q)
 \end{split}
\end{equation*}
where $u_p$, $v_p$, $u_q$ and $v_q$ are in $\oo_L \cdot \tilde{\oo}_K$. 
Setting $\omega'$ to be a uniformiser in the maximal ideal of $\tilde{\oo}_K\cdot \oo_L$, we have
\begin{equation*}
\length\left(\frac{\oo_{\cc_2 \times \oo_L}}{I_p + I_q}\right) = \min\left(\ord_{\omega'} (a u_p - a u_q), \ord_{\omega'} (a v_p - a v_q)\right). 
\end{equation*}
In the case $B = (u,v,a)$ we look at the affine patch of the blowup given by setting $a \ne 0$; the equations for $p$ and $q$ transform into
\begin{equation*}
 p' = (u-u_p, v-v_p) \;\;\; \text{ and } \;\;\; q' = (u-u_q, v-v_q), 
\end{equation*}
so 
\begin{equation*}
\begin{split}
\length\left(\frac{\oo_{\cc_3 \times \oo_L}}{I_p + I_q}\right) & = \min\left(\ord_{\omega'} (u_p -  u_q), \ord_{\omega'} ( v_p -  v_q)\right)\\ 
& = \length\left(\frac{\oo_{\cc_2 \times \oo_L}}{I_p + I_q}\right) - \ord_{\omega'} (a). 
\end{split}
\end{equation*}
In the case $B = (u,a)$ we look again at the affine patch of the blowup given by setting $a \ne 0$; the equations for $p$ and $q$ transform into
\begin{equation*}
 p' = (u-u_p, v- a v_p) \;\;\; \text{ and } \;\;\; q' = (u-u_q, v- a v_q), 
\end{equation*}
so 
\begin{equation*}
\begin{split}
\length\left(\frac{\oo_{\cc_3 \times \oo_L}}{I_p + I_q}\right) & = \min\left(\ord_{\omega'} (u_p -  u_q), \ord_{\omega'} (  a v_p -  a  v_q)\right)\\ 
& = \length\left(\frac{\oo_{\cc_2 \times \oo_L}}{I_p + I_q}\right) - (0 \text{ or } 1)\ord_{\omega'} (a), 
\end{split}
\end{equation*}
so the result follows from the fact that, since $\tilde{\oo}_K$ is unramified over $\oo_K$, we have
\begin{equation*}
 \ord_{\omega'}(a) = \ramdeg(L \cdot \tilde{K}/\tilde{K}) = \ramdeg(L/K). 
\end{equation*}
\end{proof}

\begin{proof}[Proof of Proposition \ref{prop:finite_comp}]
To prove Proposition \ref{prop:finite_comp}, we apply Lemmata \ref{lema:comp_A}, \ref{lema:comp_B}, \ref{lema:comp_C} and \ref{lema:comp_D} in that order to find that there exists $0 \le \beta \le b_\nu de\log(\#\kappa(\nu))$ such that
\begin{equation*}
 \begin{split}
\sum_{i,j} \log\left(\frac{1}{\dd(p_i,q_j)}\right) & = \frac{1}{m}\sum_{i,j}\sum_{\Omega | \nu} \log(\#\kappa(\Omega))\length_{\oo_L}\left(\frac{\oo_{\cc_1 \times_{\oo_K} \oo_L, \Omega}}{I_p + I_q}\right) \\
& = \frac{1}{m}\sum_{i,j}\sum_{\Omega | \nu} \log(\#\kappa(\Omega))\length_{\oo_L}\left(\frac{\oo_{\cc \times_{\oo_K} \oo_L, \Omega}}{I_p + I_q}\right)  + \beta\\
& = \frac{1}{m}\log(\#\kappa(\omega))\length_{\oo_L}\left(\frac{\oo_{\cc \times \oo_L}}{I_D + I_E}\right)  + \beta \\
& = \frac{1}{m}\log(\#\kappa(\omega))\length_{\oo_K}\left(\frac{\oo_{\cc}}{I_D + I_E}\right) \cdot \ramdeg(L/K)  + \beta \\
& = \log(\#\kappa(\nu))\intersect{\overline{D}}{\overline{E}}  + \beta. \\
 \end{split}
\end{equation*}
\end{proof}

\begin{proof}[Proof of Theorem \ref{thm:non-Arch}]
Let $M^+$ be the matrix from Lemma \ref{bounding Phi}, let $m_-$ denote the infimum of the entries of $M^+$ and $m_+$ their supremum. Let $b_\nu$ be the integer appearing in Proposition \ref{prop:finite_comp}. Set 
\begin{equation*}
\mathscr{B}_\nu = \left(2g^2(m_+ - m_-) + g^2 b_\nu \right) \log(\#\kappa(\nu)). 
\end{equation*}
Then the result follows from Lemma \ref{bounding Phi} and Proposition \ref{prop:finite_comp}. 
\end{proof}

\section{Archimedean results}
\label{sec:arch_abs_values}

\subsection{Defining metrics}
\label{sec:Arch}
As in the non-Archimedean setting, we will define a metric and compare the distance between divisors in this metric to the local N\'eron pairing between the divisors (more precisely, between the corresponding points on the Jacobian). 

\begin{df}\label{df:Arch_metric}
For Archimedean absolute values $\nu$ we define 
\begin{equation*}
\dd_\nu: C(K_\nu\algcl) \times C(K_\nu\algcl) \rightarrow \mathbb{R}_{\ge 0}
\end{equation*}
by
\begin{equation*}
\begin{split}
&\dd_\nu((X_p:S_p:Y_p),(X_q:S_q:Y_q)) \\
& = \min\left(1, \max\left( \abs{{x_p} - {x_q}}_{\nu}, \abs{{y_p^{g+1}} - {y_q^{g+1}}}_{\nu}\right),  \max\left( \abs{{s_p} - {s_q}}_{\nu}, \abs{{{y'_p}^{g+1}} - {{y_q'}^{g+1}}}_{\nu}\right)  \right), \\
\end{split}
\end{equation*}
where as always $x_p = X_p/S_p$ etc. 
\end{df}


\subsection{Estimates for the Archimedean distance in a special case}

\newcommand{\field}{K}
In the special case where points $p$ and $q$ in $C(K)$ are related by the hyperelliptic involution, we can easily relate the distance between $p$ and $q$ to the $y$-coordinate of $p$ (we will need this estimate in Section \ref{sec:refined_naive}):
\begin{lemma}\label{lem:archconst}
There exist computable constants $0 < \dtwo < \dthree $ such that for all non-Weierstrass points $p = (X:S:Y) \in C(\field\algcl)$, and for all Archimedean absolute values $\nu\in M^\infty_K$ on $\field$ with their unique extensions to $\field\algcl$, we have
\begin{equation*}
\dtwo \le  \dd_\nu (p, p^-)/ (2 \min(\abs{y}_\nu, \abs{y'}_\nu)) \le \dthree, 
\end{equation*}
where as usual we write $y = Y/S^{g+1}$ and $y' = Y/X^{g+1}$. 
\end{lemma}
\begin{proof}Since $M^\infty_K$ is finite, it is enough to show that such bounds can be found for one $\nu \in M_K^\infty$ at a time. 
Fix an Archimedean absolute value $\nu$. Recall that $\dd_\nu$ is the metric given in Definition \ref{df:Arch_metric}. A brief calculation (considering the two cases $\abs{y}\le\abs{y'}$ and $\abs{y} \ge \abs{y'}$) shows that
$$\frac{\dd_\nu (p, p^-)}{(2 \min(\abs{y}_\nu, \abs{y'}_\nu)) }= \min\left(1, \frac{1}{2\min(\abs{y}_\nu, \abs{y'}_\nu)}   \right). $$
Recall that $C$ is given by 
$$Y^2 = \sum_{i=0}^{2g+2} f_i X^iS^{2g+2-i}, $$
and set $a = \sqrt{\sum_i \abs{f_i}_\nu}$. Then $\abs{X/S}_\nu \le 1$ implies $\abs{y}_\nu \le a$ and $\abs{S/X}_\nu \le 1$ implies $\abs{y'}_\nu \le a$, so we find
$$\min\left(1, \frac{1}{2a}\right) \le \frac{\dd_\nu (p, p^-)}{(2 \min(\abs{y}_\nu, \abs{y'}_\nu))} \le 1. $$
\end{proof}

\subsection{Local N\'eron pairing in the Archimedean case}\label{sec:NLP_arch}
As in the non-Archimedean case, we will make use of the local N\'eron pairing to compare our metric to the local part to the N\'eron-Tate height. We recall in outline the construction of the pairing from \cite{lang1988introduction}, where more details can be found. 

Let $\nu$ be an Archimedean absolute value of $K$. Fix an algebraic closure of $K_\nu$, and view $C_\nu = C(K^{\on{alg}}_\nu)$ as a compact connected Riemann surface of positive genus and let $\mu$ denote the canonical (Arakelov) (1,1)-form $\mu$ on $C_\nu$ (as in \cite[II, \S2, page 28]{lang1988introduction}).  We write $G(-,-):C_\nu \times C_\nu \rightarrow \R_{\ge 0}$ for the exponential Green's function on $C_\nu \times C_\nu$  associated to $\mu$, and $\on{gr}$ for its logarithm. We normalise the Green's function to satisfy the following three properties. 

1) $G(p, q)$ is a smooth function on $C_\nu \times C_\nu$ and vanishes only at the diagonal. For a fixed $p \in C_\nu$, an open neighbourhood $U$ of $p$ and a local coordinate $z$ on $U$ centred at $p$, there exists a smooth function $\alpha$ such that for all $q \in U$ with $p \neq q$ we have
\begin{equation*}
\on{gr}(p,q) =\log \abs{z(q)} + \alpha(q).
\end{equation*}

2) For all $p \in C_\nu$ we have $\partial_q \overline{\partial}_q \on{gr}(p,q)^2 = 2 \pi i \mu(q)$ for $q \neq p$. 

3) For all $p \in C_\nu$, we have
\begin{equation*}
\int_{C_\nu} \on{gr}(p,q) \mu(q) = 0. 
\end{equation*}

Write $D = \sum_i a_i p_i$ and $E = \sum_j b_j q_j$ with $a_i$, $b_j \in \mathbb{Z}$ and $p_i$, $q_j \in C_\nu$ (where $D$ and $E$ are assumed to have degree 0 and disjoint support). Then the local N\'eron pairing at $\nu$ is defined by
\begin{equation*}
\NLP{D,E}_\nu = \sum_{i,j} a_i b_j \on{gr}(p_i, q_j). 
\end{equation*}

\subsection{Comparing the metric and the local N\'eron pairing}

Fix an embedding of $K$ into $\mathbb{C}$. Let $\on{gr}$ be the logarithmic Green's function on the Riemann surface $C(\mathbb{C})$ (defined using this embedding) given in Section \ref{sec:NLP_arch}. We have:


\begin{proposition}\label{local_arch}
There exists a constant $c \ge 0$ such that for all pairs of distinct points $p$, $q \in C(\mathbb{C})$, we have
$$\abs{\on{gr}(p,q) + \log \on{d}_\nu(p,q)} \le c .$$
\end{proposition}
\begin{proof}
Let $\Delta$ be the diagonal in the product $C \times_K C$. The Green's function $\on{gr}$ can be taken to be the logarithm of the norm of the canonical section of the line bundle $\mathcal{O}_{C \times C}(\Delta)$ (see \cite[4.10]{Moret-Bailly1985Metriques-permi} for details). We need to show that the functions $\on{gr}(-,-)$ and $\log \on{d}_\nu(-,-)$ differ by a bounded amount. This is easy: both functions are continuous outside the diagonal $\Delta$, and exhibit logarithmic poles along the diagonal (\cite[4.11]{Moret-Bailly1985Metriques-permi}), so their difference is bounded by a compactness argument. 
\end{proof}
The following proposition is the Archimedean analogue of theorem \ref{thm:non-Arch}, except we omit the `explicitly computable'. This makes it much easier to prove. 

\begin{proposition}\label{prop:Arch}
Given an Archimedean absolute value $\nu \in M_K^0$, there exists a constant $\mathscr{B}_\nu$ with the following property: 

Let $D = D_1 - D_2$ and $E = E_1 - E_2$ be differences of reduced divisors on $C$ with no common points in their supports, and assume that $D$ and $E$ both have degree zero. 
Write $D = \sum_i d_ip_i$, $E = \sum_j e_jq_j$, with $d_i$, $e_j \in \mathbb{Z}$ and $p_i$, $q_j \in C(\mathbb{C})$. Recall from Section \ref{sec:LNP} that $\NLP{D,E}_\nu$ denotes the local N\'eron pairing of $D$ and $E$ at $\nu$. Then
\begin{equation*}
\abs{ \NLP{D,E}_\nu - \sum_{i,j}d_ie_j\log\left(\frac{1}{\dd_\nu(p_i,q_j)} \right) } \le \mathscr{B}_\nu.
\end{equation*}

We call such a constant $\mathscr{B}_\nu$ a \emph{height-difference bound at $\nu$}. 
\end{proposition}
\begin{proof}
This follows immediately from the definition of the N\'eron local pairing and proposition \ref{local_arch}. 
\end{proof}

The key result is now:
\begin{theorem}\label{lem:rho}
There exists an algorithm which, given an Archimedean place $\nu$, will compute a height difference bound $\mathscr{B}_\nu$ at $\nu$. 
\end{theorem}
The author is aware of at least $2$ proofs of this result. The first was given in \cite{holmesPhDThesis}; it begins by analysing the case were the points in the support of $D$ and $E$ are not too close together using an explicit formula from \cite{holmes2010canonical} for the Green's function in terms of theta functions, together with explicit bounds on the derivatives of theta functions. The case where some points in the support are close together is handled by a `hands-on' computation of how the Green's function and theta functions behave under linear equivalence of divisors. The proof occupies 33 pages. The second proof was given in a previous version of this paper \cite{first_version_of_this}; it uses Merkl's theorem \cite{couveignes2011computational}, and requires 13 pages. The problem with these approaches is that they will be hard to implement, and more importantly will give extremely large bounds - with Merkl's theorem terms like $\exp(4800g^2)$ appear in the difference between the exponential heights, making this entirely impractical for calculations. Problems with methods coming from numerical analysis are discussed in the introduction. 

What is needed is an algorithm which is practical to implement and gives small, rigorous bounds. It seems that at the time of writing no such algorithm is known (though note that Silverman \cite{MR1035944} essentially gives an explicit value for $\mathscr{B}_\nu$ in the case where $g = 1$). Since the existing algorithms are lengthy to write down and have no practical application (due to the size of the bounds they produce), we will not describe them in detail here. 

\section{The first na\"ive height}
\label{sec:first_naive}

\newcommand{\Hnaive}{\operatorname{H}^\text{n}}
\newcommand{\hnaive}{\operatorname{h}^\text{n}}
\newcommand{\ipone}{E}
\newcommand{\iptwo}{E'}
\begin{assumption}\label{assumptions on C}
In this section we will for the first time require that $\#M_K^\infty \le 1$ (so $\on{char}K >0$ or $K = \mathbb{Q}$). We also assume that the curve $C$ has a rational Weierstrass point, and we move a rational Weierstrass point of $C$ to lie over $s=0$, so that the affine equation for $C$ has degree $2g+1$. We denote this point by $\infty$. We further assume that there is no Weierstrass point $d$ with $X_d = 0$. None of these assumptions are essential, but they simplify the exposition. 
\end{assumption}

\begin{remark}\label{rk:inf_ass}
The assumption that $\#M_K^\infty \le 1$ is to ensure the existence of divisors $E$ and $E'$ in the next definition. To treat the general case, one may have to use several pairs of divisors $E$ and $E'$, one for each Archimedean place of $K$. The comparisons of the heights will then become more involved. 
\end{remark}

\begin{df}
\label{def:definition_of_naive_height}
If K has positive characteristic, set $\mu = 1$. Otherwise, let $\mu := \frac{1}{3}\min_{w, w'} \dd_\nu(w,w')$ where the minimum is over pairs of distinct Weierstrass points of $C$, and $\nu$ is the Archimedean absolute value. 

Given a rational point $p$ of the Jacobian $\on{Jac}_C$ of $C$, write $p = [D - \deg(D) \infty]$ where $D$ is a reduced divisor on $C$ such that the coefficient of $\infty$ in $D$ is zero (such a $D$ is unique). If the support of $D$ contains any Weierstrass points, replace $D$ by the divisor obtained by subtracting them off. Let $d$ denote the degree of the resulting divisor $D$. 

Choose once and for all a pair of degree-$d$ effective divisors $\ipone$ and $\iptwo$ with disjoint support, supported on Weierstrass points away from $\infty$, such that no point in the support of $D$ is within Archimedean distance $\mu$ of any point in the support of $\ipone$ or $\iptwo$. The existence of such divisors is clear since there are $2g+1$ Weierstrass points away from $\infty$ and reduced divisors have degree at most $g$. 

Let $D^-$ denote the image of $D$ under the hyperelliptic involution. Let $L/K$ denote the minimal field extension over which $D$, $\ipone$ and $\iptwo$ are pointwise rational. Over $L$, we write $D = \sum_i d_i$, $\ipone = \sum_i q_i$ and $\iptwo = \sum_i q_i'$. Given an absolute value $\nu$ of $L$, define
 \begin{equation*}
\dd_\nu(D-\ipone,D^- - \iptwo) := \prod_{i,j} \frac{\dd_\nu(p_i,p_j^-) \dd_\nu(q_i,q_j')}{\dd_\nu(p_i, q_j') \dd_\nu(p_j^-,q_i)}. 
\end{equation*}

Define the height $\Hnaive: \on{Jac}_C(K) \rightarrow \mathbb{R}_{\ge 1}$ by 
\begin{equation}\label{eq:asd}
\Hnaive(p) = \left( \prod_{\nu \in M_L} \frac{1}{\dd_\nu(D-\ipone,D^- - \iptwo)}\right)^{\frac{1}{[L:K]}}. 
\end{equation}
We define a logarithmic na\"ive height by $\hnaive(p) = \log(\Hnaive(p))$. 
\end{df}
\noindent Note that $\dd_\nu(D-\ipone,D^- - \iptwo) = 1$ for all but finitely many absolute values $\nu$, and so the product in Equation \eqref{eq:asd} is finite.

Write $\alpha\colon \operatorname{Div}^0(C) \rightarrow \on{Jac}_C(K)$ for the usual map. The crucial result which allows us to relate our na\"ive height to the N\'eron-Tate height is:
\begin{theorem}[Faltings, Hriljac]\label{FH}
Let $D_1$ and $D_2$ be two divisors of degree zero on $C$ with disjoint support. Suppose $D_1$ is linearly equivalent to $D_2$. Then
\begin{equation*}
\sum_{\nu \in M_K}\NLP{D_1,D_2}_\nu = -\hat{\hh}(\alpha(D_1))
\end{equation*}
where $\hat{\hh}$ denotes the N\'eron-Tate height function with respect to twice the theta-divisor. 
\end{theorem}
\begin{proof}
See \cite{faltings1984calculus} or \cite{hriljac_thesis} for the case where $K$ is a number field. The same proof works when $K$ is a global field as has been remarked by a number of authors, see e.g. \cite{SteffenComputation}. 
\end{proof}

\begin{theorem}
\label{thm:tilde_height_diff}
There exists a computable constant $\done \ge 0$  such that for all $p \in \on{Jac}_C(K)$ we have 
\begin{equation*}
\abs{\hat{\hh}(p) - \hnaive(p)} \le\done. 
\end{equation*}
\end{theorem}
\begin{proof}

For each absolute value $\nu$ of $K$, let $\mathscr{B}_\nu$ be the real number defined in Theorem \ref{thm:non-Arch} for $\nu$ non-Archimedean, and in Proposition \ref{prop:Arch} for $\nu$ Archimedean. Note that $\mathscr{B}_\nu= 0$ for $\nu$ a non-Archimedean absolute value of good reduction for $C$. Define 
\begin{equation*}
\done := \sum_{\nu \in M_K} \mathscr{B}_\nu.  
\end{equation*}
Let $D$, $D^-$, $\ipone$, $\iptwo$ be the divisors associated to $p$ as in Definition \ref{def:definition_of_naive_height}. Recall from Section \ref{sec:LNP} that $[-,-]_\nu$ denotes the local N\'eron pairing at $\nu$ between two divisors of degree zero and with disjoint supports. Then by Theorem \ref{thm:non-Arch} and Proposition \ref{prop:Arch} we have that
\begin{equation*}
\abs{\sum_{\nu \in M_K} \NLP{D- \ipone,D^- - \iptwo}_\nu - \hnaive(p)} \le \done. 
\end{equation*}
Now we will use Theorem \ref{FH} to compare $\sum_{\nu \in M_K} \NLP{D- \ipone,D^- - \iptwo}_\nu$ to $\hat{\hh}(p)$; in fact, we will show they are equal. First, a little more notation: write 
\begin{equation*}
[-,-] = \sum_{\nu \in M_K}[-,-]_\nu, 
\end{equation*}
(the sum of the local N\'eron pairings). This pairing is a-priori only defined for degree-zero divisors with disjoint support, but it respects linear equivalence by \cite[IV, Theorem 1.1]{lang1988introduction}, and hence extends to a bilinear pairing on the whole of $\on{Div}^0(C)$, and moreover factors via $\on{Jac}_C(K)$. Write 
\begin{equation*}
\Span{\Span{-,-}}\colon \on{Jac}_C(K) \times \on{Jac}_C(K) \rightarrow \mathbb{R}
\end{equation*}
for the N\'eron-Tate height pairing (so $\Span{\Span{x,x}} = -\hat{\hh}(x)$ for all $x \in \on{Jac}_C(K)$). Theorem \ref{FH} then tells us that 
\begin{equation*}
[F,F] = -\Span{\Span{\alpha(F), \alpha(F)}}
\end{equation*}
for every degree-zero divisor $F$ on $C$, but since a bilinear form is determined by its restriction to the diagonal we find that
\begin{equation*}
[F,F'] = -\Span{\Span{\alpha(F), \alpha(F')}}
\end{equation*}
for every pair $F$, $F'$ of degree-zero divisors on $C$.

Write $\tilde{p} = \alpha(D-E)$, and $q = \alpha(D^-  - E')$. Then there exist 2-torsion points $\sigma$, $\tau \in \on{Jac}_C(K)$ such that 
\begin{equation*}
\tilde{p} = p + \sigma \;\;\; \text{and}\;\;\; -q = p + \tau. 
\end{equation*}

By the above discussion, we know that
\begin{equation*}
\begin{split}
\sum_{\nu \in M_K} \NLP{D- \ipone,D^- - \iptwo}_\nu & = [D-E, D^- - E']\\
& = \Span{\Span{\alpha(D- \ipone), \alpha(D^- - \iptwo)}}\\
& = \Span{\Span{p + \sigma, -p - \tau}}\\
& = \Span{\Span{p,-p}} + \Span{\Span{p,-\tau}} + \Span{\Span{ \sigma, -p}} + \Span{\Span{\sigma, -\tau}}. 
\end{split}
\end{equation*}
Now since $\Span{\Span{-,-}}$ is bilinear, it vanishes whenever either of the inputs is a torsion point, so we see that 
\begin{equation*}
\sum_{\nu \in M_K} \NLP{D- \ipone,D^- - \iptwo}_\nu = \Span{\Span{p,-p}} = \hat{\hh}(p)
\end{equation*}
as desired. 
%
%
%
%
\end{proof}

\section{Refined na\"ive heights}
\label{sec:refined_naive}	

We introduce two new na\"ive heights which are each in turn simpler to compute, and we bound their difference from the N\'eron-Tate height. We will be able to compute the finite sets of points of bounded height with respect to the last of these heights.

\newcommand{\hheart}{\hh^\heartsuit}
\newcommand{\mheart}{M^{\heartsuit}}
\newcommand{\hdagger}{\hh^\dagger}
\newcommand{\mdagger}{M^{\dagger}}

\begin{df}
 Given $p \in \on{Jac}_C(K)$, let $D = \sum_{i=1}^d p_i$ denote the corresponding divisor over some finite $L/K$ as in Definition \ref{def:definition_of_naive_height}, and write $p_i = (x_{p_i}, y_{p_i})$. Then set
\begin{equation*}
 \hheart(p) = \sum_{i=1}^d \hh(x_{p_i}), 
\end{equation*}
(where $\hh$ is the absolute usual height on an element of a global field as specified in Section \ref{sec:outline}) and set 
\begin{equation*}
 \hdagger(p) = \hh\left(\prod_{i=1}^d (x-x_{p_i})\right), 
\end{equation*}
where the right hand side is the height of a polynomial, which by definition is the height of the point in projective space whose coordinates are given by its coefficients. 
\end{df}
We will give computable upper bounds on $\hheart  - \hnaive$ and on $\abs{\hheart - \hdagger}$.

\newcommand{\pdist}[1]{\Span{#1}}

\begin{df}
Let $L/K$ be a finite extension, and let $p \ne q \in C(L)$ be distinct points. Set
\begin{equation*}
 \pdist{p,q}_L = \frac{-1}{[L:K]}\log\prod_{\nu \in M_{L}} \dd_\nu(p,q). 
\end{equation*}
\end{df}

\begin{lemma}
\label{lem:p_p_bar}
There exists a computable constant $\dfour$ with the following property: 

let $L/K$ be a finite extension, and let $p = (X:S:Y) \in C(L)$ be a non-Weierstrass point. Then
\begin{equation*}
\abs{\pdist{p, p^-}_L -  (g+1)\hh(X/S)} \le \dfour
\end{equation*}
\end{lemma}
\begin{proof}
For $\abs{-}_\nu$ non-Archimedean, we have that if $\abs{X}_\nu \le \abs{S_\nu}$ then $\dd_\nu(p,p^-) = \abs{2Y/S^{g+1}}_\nu$, and if $\abs{S}_\nu \le \abs{X}_\nu$ then $\dd_\nu(p,p^-) = \abs{2Y/X^{g+1}}_\nu$. Hence for non-Archimedean $\nu$ we obtain
\begin{equation*}
 \dd_\nu(p,p^-) = \abs{2Y}_\nu \min(1/\abs{X}^{g+1}_\nu, 1/\abs{S}^{g+1}_\nu). 
\end{equation*}
By Lemma \ref{lem:archconst},  for Archimedean $\nu$ we have computable $0 < \dtwo < \dthree $ such that
\begin{equation*}
\dtwo < \dd_\nu(p, p^-) /  \min(\abs{2Y/X^{g+1}}_\nu, \abs{2Y/S^{g+1}}_\nu) < \dthree. 
\end{equation*}
Hence
\begin{equation*}
\prod_{\nu \in M_L^\infty} 1/\dthree \le 
\frac{\prod_{\nu \in M_L}1/\dd_\nu(p,p^-)}{ \prod_{\nu \in M_L} \abs{2Y}^{-1}_\nu \prod_{\nu \in M_L} \max(\abs{X}_\nu, \abs{S}_\nu)^{g+1} } \le  \prod_{\nu \in M_L^\infty} 1/\dtwo. 
\end{equation*}
Now $\prod_{\nu \in M_L^\infty} {\dtwo}^{-1/[L:K]}$ is bounded uniformly in $L$, and similarly for $\dthree$. Finally, note 
\begin{equation*}
 \left(\prod_{\nu \in M_L} \abs{2Y}^{-1}_\nu\right) \left(\prod_{\nu \in M_L} \max(\abs{X}_\nu, \abs{S}_\nu)\right)^{g+1} 
= H(x/s)^{[L:K](g+1)}. 
\end{equation*}
\end{proof}

Recall that in Definition \ref{def_lambda} we defined a constant $\lambda_\nu$ for each non-Archimedean absolute value $\nu$ of $K$, and that these take the value 1 for all but finitely many $\nu$. 
\begin{df}
We set $$\deltalambda = (2g+3/2)\sum_{\nu \in M^0_K} \log \lambda_\nu. $$ 
\end{df}
%

\begin{lemma}
\label{lem:GCD}Let $L/K$ be a finite extension, and let $p$, $w \in C(L)$ with $p \neq w$ be such that $s_p \ne 0$ and $w$ is a Weierstrass point with $s_w \ne 0$. Then
\begin{equation*}
 -\sum_{\nu \in M_L^0}\log\dd_\nu(p, w) \le [L:K]\left(\frac{1}{2}\hh(x_p-x_w) +  \deltalambda   \right).
\end{equation*} 
\end{lemma}
\begin{proof}
The right hand side naturally decomposes as
\begin{equation*}
\sum_{\nu \in M_L} \left( \frac{1}{2}\log^+\abs{X_p-X_w}^{-1}_\nu + (2g+3/2)\log \lambda_{\nu'} \right),   
\end{equation*}
where $\nu'$ is the absolute value on $K$ which extends to $\nu$. Now it is clear that
\begin{equation*}
 \sum_{\nu \in M^{\infty}_L} \frac{1}{2}\log^+\abs{x_p-x_w}^{-1}_\nu   \ge 0, 
\end{equation*}
so it suffices to prove that for each non-Archimedean $\nu$ we have
\begin{equation*}
 -\log(\dd_\nu(p,w)) \le \frac{1}{2}\log^+\abs{x_p-x_w}^{-1}_\nu  + (2g+3/2)\log\lambda_{\nu'}. 
\end{equation*}
This is exactly the statement of Lemma \ref{lem:GCD_non_Arch}. 
%
 \end{proof}

\begin{lemma}
\label{lem:height_sum} Let $L/K$ be a finite extension, and let $\on{H}$ denote the usual exponential height on $L$. Let $x_1$, $x_2 \in L$. Then $\HH(x_1 + x_2) \le 2^{\#\on{M}_K^\infty}\HH(x_1)\HH(x_2)$. 
\end{lemma}
\begin{proof} Omitted. 
\end{proof}

\begin{lemma}
\label{lem:phimu}
There exists a computable constant $\dfive$ with the following property:

 Let $L/K$ be a finite extension, and let $p$, $w \in C(L)$ such that $s_p \ne 0$ and $w$ is a Weierstrass point with $s_w \ne 0$. Suppose also that $\dd_\nu(p,w) \ge \mu$ for all Archimedean $\nu$ (where $\mu$ is the constant from Definition \ref{def:definition_of_naive_height}). Then 
\begin{equation*}
\pdist{p,w}_L \le \frac{1}{2}\hh(x_p) + \dfive. 
\end{equation*} 
\end{lemma}

\begin{proof}
From Lemma \ref{lem:GCD} we see that 
\begin{equation*}
\pdist{p,w}_L \le  \frac{1}{2}\hh(x_p-x_w)  + \deltalambda - \log(\mu). 
\end{equation*}
Now by Lemma \ref{lem:height_sum}, we have
\begin{equation*}
 \hh(x_p - x_w) \le \hh(x_p) + \hh(x_w) + \#M_K^\infty\log(2). 
\end{equation*}
We define
\begin{equation*}
 {\dfive}(w) =  - \log(\mu) +\frac{1}{2} \hh(x_w) + \frac{\#M_K^\infty}{2} \log(2) + \deltalambda. 
\end{equation*}
Then we find that for all $L$ and $p$ as in the statement, we have 
\begin{equation*}
\pdist{p,w}_L \le \frac{1}{2}\hh(x_p) + {\dfive}(w). 
\end{equation*} 
Finally, there are only finitely many Weierstrass points, so setting $\dfive = \max_w {\dfive}(w)$, we are done. 
\end{proof}

\begin{lemma}
\label{lem:quadratic}There exists a computable constant $\dsix$ such that the following holds. 

Given $p \in \on{Jac}_C(K)$, let $D$, $\ipone$ and $\iptwo$ denote the divisors given in Definition \ref{def:definition_of_naive_height}. Let $L/K$ be the minimal finite extension such that $D$, $\ipone$ and $\iptwo$ are all pointwise rational over $L$. We write
\begin{equation*}
  D  = \sum_{i=1}^d p_i \;\; ,\;\; \ipone = \sum_{i=1}^d q_i \;\; , \;\; \iptwo  = \sum_{i=1}^d q'_i. 
\end{equation*}
Then
\begin{equation*}
 \hnaive(p) \ge \sum_{i=1}^d \left( \pdist{p_i, p_i^-}_L - \sum_{j=1}^d \pdist{p_i, q_j}_L - \sum_{j=1}^d \pdist{p_i, q'_j}_L\right) + \dsix, 
\end{equation*}
where $p_i^-$ is the image of $p_i$ under the hyperelliptic involution. 
\end{lemma}

\begin{proof}
Recall that
\begin{equation*}
 \hnaive(p) = \sum_{i,j=1}^d  \pdist{p_i, p_j^-}_L +  \sum_{i,j=1}^d  \pdist{q_i, q'_j}_L  - \sum_{i,j=1}^d \pdist{p_i, q_j}_L - \sum_{i, j=1}^d \pdist{p^-_i, q'_j}_L. 
\end{equation*}
Since the $q_i$ and $q_i'$ are distinct Weierstrass points we easily bound $\sum_{i,j=1}^d  \pdist{q_i, q'_j}_L$. 

It remains to find a lower bound on the terms $\pdist{p_i, p_j^-}_L$ for $i \neq j$. Note that $d_\nu$ is bounded above by $1$ for all $\nu$, hence $\pdist{p_i, p_j^-}_L \ge 0$. 
\end{proof}

\begin{lemma}
\label{lem:cool}
There exists a computable constant $\dseven$ such that in the setup of Lemma \ref{lem:quadratic} we have
\begin{equation*}
 \hnaive(p) \ge \sum_{i=1}^d h(x_{p_i}) + \dseven. 
\end{equation*}
\end{lemma}
\begin{proof}
In Lemma \ref{lem:quadratic} we showed 
\begin{equation*}
 \hnaive(p) \ge \sum_{i=1}^d \left( \pdist{p_i, p_i^-}_L - \sum_{j=1}^d \pdist{p_i, q_j}_L - \sum_{j=1}^d \pdist{p_i, q'_j}_L\right) + \dsix. 
\end{equation*}
In Lemma \ref{lem:p_p_bar} we showed (using that the $p_i$ are never Weierstrass points) that for some computable $\dfour$ we have
\begin{equation*}
\abs{\pdist{p_i, p_i^-}_L -  (g+1)\hh(x_{p_i})} \le \dfour. 
\end{equation*}
In Lemma \ref{lem:phimu} we showed that
\begin{equation*}
\pdist{p_i,q_j}_L \le \frac{1}{2}\hh(x_{p_i}) + \dfive,  
\end{equation*} 
and similarly for $q_j'$. 

Combining these, we see using $d \le g$ that for each $i$
\begin{equation*}
\begin{split}
 \pdist{p_i, p_i^-}_L - \sum_{j=1}^d \pdist{p_i, q_j}_L - \sum_{j=1}^d \pdist{p_i, q'_j}_L & \ge (g+1)\hh(x_{p_i}) - 2\sum_{j=1}^d \frac{1}{2}\hh(x_{p_i})- \dfour + 2d\dfive  \\
& = ((g+1) - 2d\frac{1}{2})\hh(x_{p_i}) - \dfour + 2d\dfive  \\
& \ge \hh(x_{p_i}) - \dfour + 2d\dfive. \\
\end{split}
\end{equation*}
from which the result follows. 
\end{proof}

\begin{theorem}
\label{thm:heart_height_diff}
 There exists a computable constant $\deight$ such that for all $p \in A(K)$ we have
\begin{equation*}
 \hat{\hh}(p) + \deight \ge \hheart(p). 
\end{equation*}
\end{theorem}
\begin{proof}
Set $\deight = \done + \dseven$. The result follows from Theorem \ref{thm:tilde_height_diff} and Lemma \ref{lem:cool}. 
\end{proof}

\begin{lemma}
Fix a finite extension $L/K$. Given $a_1, \ldots, a_n \in L$, set $\psi_n = \prod_{i=1}^n (t-a_i) \in L[t]$. If $\on{char} K >0$ then $\hh(\psi_n) = \sum_{i=1}^n \hh(a_i)$, otherwise 
\begin{equation*}
 \abs{\hh(\psi_n) - \sum_{i=1}^n \hh(a_i)} \le n \log 2 
\end{equation*}
We summarise this by writing
\begin{equation*}
 \abs{\hh(\psi_n) - \sum_{i=1}^n \hh(a_i)} \le (n \log 2 ) \delta_{\on{char} K}
\end{equation*}
\end{lemma}
\begin{proof}
\cite[Theorem VIII.5.9]{Silverman2009The-arithmetic-}
\end{proof}

\begin{corollary}
 For all $p \in A(K)$ we have
\begin{equation*}
\abs{\hheart(p) - \hdagger(p)} \le  (g\log 2)\delta_{\on{char} K}.
\end{equation*}
\end{corollary}
\begin{df}
Given a real number $B$, we define 
\begin{equation*}
\hat{M}(B) := \{p \in A(K) | \hat{\hh}(p) \le B \} 
\end{equation*}
and
\begin{equation*}
\mdagger(B) := \{p \in A(K) | \hdagger(p) \le B \} .
\end{equation*}
\end{df}

The main result of this paper is the following. 
\begin{corollary}
\label{thm:dagger_M}
Let $B \in \mathbb{R}$. Let $B' = B+ \deight + (g\log 2)\delta_{\on{char} K}$. Then for all real numbers $B$ we have
\begin{equation*}
\hat{M}(B)  \subset \mdagger\left(B'\right). 
\end{equation*}
Moreover, the finite set $\mdagger(B')$ is computable, and hence by results in \cite{holmes2010canonical} so is the finite set $\hat{M}(B)$. 
\end{corollary}

\begin{proof}
The inclusion follows from the results above. We describe one algorithm to compute $\mdagger(B)$. 

1) Let $S$ be the finite set of all polynomials $\prod_{i=1}^d (x-a_i)$, for $d \le g$, of height up to $B$. 


2) It suffices to determine for each $a \in S$ whether $a$ is the `$x$-coordinate polynomial' of a divisor in Mumford representation (see \cite[III, Proposition 1.2]{mumfordTataII}); in other words, whether there exists another univariate polynomial $b$ such that $(a,b)$ satisfy the properties of a Mumford representation. This corresponds to checking whether the polynomial $f-a^2$ has a factor of degree less that $\deg a$, which is widely implemented. 
\end{proof}

\begin{remark}
How hard is it to check whether such a polynomial $f-a^2$ has a factor of degree less that $\on{deg}a$? Note that $\on{deg}f-a^2 = 2g+1$, and in general $\on{deg}a = g$. Based on this, it seems reasonable that the difficulty of testing for such a factor will be somewhere in between the difficulty of factoring a polynomial of degree $2g+1$ and that of factoring a polynomial of degree $2g-1$ (since in the latter case, irreducibility is equivalent to not having a factor of degree at most $g-1$). 

In practice, the integer $g$ will usually be very small (genera 3 and 4 are the obvious cases to treat), but we will have a huge number of polynomials $a$ to run through.  Because of this, rather than looking at the time taken to check for factors of degree $<g$ in one polynomial, it is more useful to look at how efficiently we can check this for large families of $a$. One method to rapidly exclude many possible values of $a$ from the search region is by reduction modulo small primes, followed by the `Chinese remainder theorem'. The proportion of polynomials of degree $2g+1$ over a finite field $\mathbb{F}_p$ which are irreducible is approximately 
\begin{equation*}
\frac{1}{2g+1},
\end{equation*}
and the proportion without a factor of degree less than $g$ is approximately 
\begin{equation*}
\frac{1}{2g+1} + \frac{1}{g^2} + \frac{1}{g^2 + g}. 
\end{equation*}
As such, at least from this point of view, we cannot expect very substantial computation savings from the fact that we need only exclude factors of degree less than $g$ (instead of computing the whole factorisation). 
\end{remark}

\section{A worked example}\label{sec:e.g.}
Given a prime number $p$, we fix a proper multi-set of absolute values $M_{{\mathbb{F}_p(t)}}$ by requiring it to contain exactly once the unique $\abs{-}_t$ such that $\abs{t}_t = p^{-1}$. We begin by bounding the difference between the first and final na\"ive heights for a certain infinite family of curves. First we define the infinite family:

\begin{df}\label{def:example}
Fix an integer $g > 0$. Let $p$ be a prime number not dividing $2(2g+1)$, and let $K = \mathbb{F}_p(t)$. Let $C$ denote the hyperelliptic curve with affine equation
\begin{equation*}
y^2 = x^{2g+1} + t. 
\end{equation*}
\end{df}
\begin{proposition}\label{prop:eg}
For all points $q \in \on{Jac}_C(K)$, we have 
\begin{equation*}
\hnaive(q) +  \frac{g(8g^2 + 15g + 4)\log p}{2g+1} \ge \hheart(q) = \hdagger(q). 
\end{equation*}
\end{proposition}
\begin{proof}
We will need to compute various heights and valuations of elements of $K$ and extensions. Fix a primitive $(2g+1)$-th root $\zeta$ of $1$ in $K\algcl$. Write $f = x^{2g+1} + t$, and write $\alpha_0, \cdots, \alpha_{2g}$ for the roots in $K\algcl$ of $f$, ordered such that $\alpha_n = \alpha_0\zeta^n$. For all absolute values $\nu \in M_K$, we have $\abs{\zeta}_\nu = 1$ and hence for all $n$ we have 
\begin{equation*}
\abs{\alpha_n}_\nu = \abs{\alpha_0}_\nu = \abs{t}_\nu^{1/{2g+1}}. 
\end{equation*}
Now $\abs{t}_{t} = p^{-1}$ and $\abs{t}_{1/t} = p$, and $\abs{t}_\nu = 1$ for all other $\nu \in M_K$. From this we deduce that $\hh(t) = \log p$ and for all $n$ that $\hh(\alpha_n) = (\log p )/(2g+1)$. Noting that $\alpha_n - \alpha_m = \alpha_0(\zeta^n - \zeta^m)$, we have for all $n \neq m$ and $\nu \in M_K$ that $\abs{\alpha_n - \alpha_m}_\nu = \abs{\alpha_0}_\nu$. From this we deduce that for all pairs of distinct Weierstrass points $w_i \neq w_j$, we have
\begin{equation*}
\Span{w_i,w_j}_L = \frac{2 \log p}{2g+1}, 
\end{equation*}
independent of the field $L$. 

 Since $K$ has no Archimedean absolute values we immediately see that we may take $\dtwo = \dthree = \dfour = 0$. We have $\lambda_\nu = 1$ for all $\nu$ apart from $\nu = (t)$ and $\nu = (1/t)$, where we have $\lambda_\nu = p^{1/2g+1}$. From this we see $$\deltalambda = \frac{(4g+3) \log p}{2g+1}. $$ We have
\begin{equation*}
\dfive = \frac{1}{2}\max_{n}\hh(\alpha_n) +  \deltalambda = \frac{\log p}{4g+2} +   \frac{(4g+3) \log p}{2g+1}, 
\end{equation*}
and since
\begin{equation*}
\sum_{w\neq w'} \Span{w, w'}_L = 4g\log p
\end{equation*}
(the sum is over distinct points $w$, $w'$ in $W \setminus \{\infty\}$) we may take
\begin{equation*}
\dsix = 4g\log p. 
\end{equation*}
Finally we see $\dseven = 2g^2\dfive  + \dsix$, and the result follows.  
\end{proof}

Finally, for three members of this family of curves, we will bound the difference between the N\'eron-Tate height and the na\"ive heights. This requires constructing a regular model of the curve, which we do in {\tt MAGMA} using Steve Donnelly's `regular models' function. First we give two examples with small genus over small fields, to illustrate the sizes of the bounds, and then we give an example in higher genus, to illustrate that the method to find bounds remains practical. 

\begin{theorem}\label{thm:eg}Let $p = 3$ and $g = 2$, and let $C$ be as in Definition \ref{def:example}. Then for all points $q \in \on{Jac}_C(K)$, we have 
\begin{equation*}
\hat{\hh}(q) + 86\log 3 \ge \hheart(q) = \hdagger(q). 
\end{equation*}

Let $p = 5$ and $g = 4$, and let $C$ be as in Definition \ref{def:example}. Then for all points $q \in \on{Jac}_C(K)$, we have 
\begin{equation*}
\hat{\hh}(q) + 417\log 5 \ge \hheart(q) = \hdagger(q). 
\end{equation*}

Let $p = 101$ and $g = 11$, and let $C$ be as in Definition \ref{def:example}. Then for all points $q \in \on{Jac}_C(K)$, we have 
\begin{equation*}
\hat{\hh}(q) + 5790\log 101 \ge \hheart(q) = \hdagger(q). 
\end{equation*}
\end{theorem}
\begin{proof}
We give details for the genus 11 example, the others are similar. {\tt MAGMA} code for the computations for all three curves can be obtained by downloading the arXiv source files for this paper. 

Let $u$, $t$ be coordinates on $B_K = \mathbb{P}^1_{\mathbb{F}_{101}}$ with $u = 1/t$. Applying Proposition \ref{prop:eg}, it is enough to compute the constants $\mathscr{B}_\nu$ from Theorem \ref{thm:non-Arch}. 
The model given by 
\begin{equation*}
uY^2 = uSX^{2g+1} + tS^{2g+2}
\end{equation*}
in weighted projective space $\mathbb{P}(1,1,g+1)$ over $\base_K$ is regular except over $u=0$, and moreover all fibres outside $u=0$ are irreducible. Hence $\mathscr{B}_\nu = 0$ whenever $\nu$ does not correspond to the prime $(u)$. 

Next we use {\tt MAGMA} to compute the regular model of $C$ over $(u)$. We rearrange the equation
$$uy^2 = ux^{23} + 1$$
to $\tilde{y}^2 = u\tilde{x}^{23} + u^{23}$, absorbing $u$ into $\tilde{x}$ and $u^{g+1}$ into $\tilde{y}$ (this process is equivalent to performing $1$ blow up at a closed point and $g$ blowups along smooth curves, for a total of $g+1 = 12$ consecutive blowups at smooth centres. The fibre over $u$ is now irreducible, and the whole fibre is in the centre of the last blowup. 

Now the equation is in a form where we can plug it into {\tt MAGMA}, which yields a regular model after $68$ blowups at smooth centres; the longest chain of consecutive blowups used by {\tt MAGMA} has length 7 (I am grateful to the anonymous referee for the code to compute this). Hence $19 = 7 + 12$ is the longest chain of consecutive blowups at smooth centres used (this number becomes 12 in the genus 4 case and 10 in genus 2). This regular model has 49 irreducible components in its special fibre (21 in the genus 4 case, 13 in genus 2), and the Moore-Penrose pseudo-inverse of its $49 \times 49$ intersection matrix has maximum entry $4.102\cdots$ and minimum entry $-8.076\cdots$. As a result, we find that 
\[
\begin{split}
\mathscr{B}_{(u)} & = (2g^2(4.102\cdots + 8.076\cdots) + 19 g^2)\log 101\\
& = 5246.07\cdots \log 101. \\
\end{split}
\]

Proposition \ref{prop:eg} yields a bound of
$$\frac{11(8(11^2) + 15\cdot 11 + 4)}{23} = 543.78\cdots$$
from which the result follows. 
\end{proof}

\begin{remark} 
The computations for Theorem \ref{thm:eg} took under 60 seconds to perform (and could have been done by hand with reasonable patience for genus 2). It is clear that, with the methods developed in this paper, the bottleneck is now searching for points of bounded na\"ive height, not finding a bound. As such, it would be very useful to improve the bounds given in these examples, but there seems little point in speeding up the algorithm to compute the bounds. 
\end{remark}


\bibliographystyle{alpha} 

\begin{thebibliography}{{{Mue}}11}

\bibitem[And02]{anderson2002edited}
{\scshape Anderson,~G.~W.}
\newblock Edited 4-theta embeddings of jacobians.
\newblock {\em Michigan Math J.}, \textbf{52}, Issue 2 (2004), 303-339. \mrev{2069803} (2005d:14044), \zbl{1065.14035}. 

\bibitem[Bru13]{bruin2013}
{\scshape Bruin,~P. }
\newblock Bornes optimales pour la diff\'erence entre la hauteur de Weil et la hauteur de N\'eron-Tate sur les courbes elliptiques sur $\overline{\mathbb{Q}}$. 
\newblock {\em Acta Arith.}, {\bf160}:385-397 (2013). \mrev{3119786}, \zbl{1287.11083}. 

\bibitem[BCP97]{MAGMA}
{\scshape Bosma, W.; Cannon, J.; Playoust, C. }
\newblock The {M}agma algebra system. {I}. {T}he user language.
\newblock {\em J. Symbolic Comput.}, {\bf24(3-4)}:235--265, (1997).
\newblock Computational algebra and number theory (London, 1993).\mrev{1484478}, \zbl{0898.68039}. 

\bibitem[BMS{\etalchar{+}}08]{integral_points_on_hyp}
{\scshape Bugeaud, Y.; Mignotte, M.; Siksek, S.; Stoll, M.; Tengely, Sz.}
\newblock Integral points on hyperelliptic curves.
\newblock {\em Algebra and Number Theory}, {\bf 2}:859--885, (2008). \mrev{2457355} (2010b:11066), \zbl{1168.11026}. 


\bibitem[CE{\etalchar{+}}11]{couveignes2011computational}
{\scshape Couveignes, J.M.; Edixhoven, B. et~al.}
\newblock Computational aspects of modular forms and galois representations. {\em Annals of Mathematics Studies 176}
\newblock (2011). \mrev{2849700}, \zbl{1216.11004}. 

\bibitem[CGO84]{MR775681}
{\scshape Cossart, V.;Giraud, J.;Orbanz, U.}
\newblock {\em Resolution of surface singularities}, {\bf 1101} of {\em
  Lecture Notes in Mathematics}.
\newblock Springer-Verlag, Berlin (1984).
\newblock With an appendix by H. Hironaka. \mrev{0775681} (87e:14032), \zbl{0553.14003}. 

\bibitem[CPS06]{MR2197860}
{\scshape Cremona, J.~E.; Prickett, M.; Siksek, S.}
\newblock Height difference bounds for elliptic curves over number fields.
\newblock {\em J. Number Theory}, {\bf 116(1)}:42--68 (2006). \mrev{2197860} (2006k:11121), \zbl{1162.11032}. 

\bibitem[Dem68]{MR0227166}
{\scshape Dem{\cprime}janenko, V. A.}
\newblock An estimate of the remainder term in {T}ate's formula.
\newblock {\em Mat. Zametki}, {\bf 3}:271--278 (1968). \mrev{0227166}, (37 \#2751) \zbl{0767.11056}. 



\bibitem[DP02]{david2002minorations}
{\scshape David, S.; Philippon, P. }
\newblock Minorations des hauteurs normalis{\'e}es des sous-vari{\'e}t{\'e}s de
  vari{\'e}t{\'e}s abeliennes ii.
\newblock {\em Commentarii Mathematici Helvetici}, {\bf 77(4)}:639--700 (2002). \mrev{1949109} (2004a:11055), \zbl{1030.11026}. 

\bibitem[Fal84]{faltings1984calculus}
{\scshape Faltings, G.}
\newblock {Calculus on arithmetic surfaces}.
\newblock {\em The Annals of Mathematics}, {\bf 119(2)}:387--424 (1984). \mrev{0740897} (86e:14009), \zbl{0559.14005}. 

\bibitem[Fly93]{MR1219694}
{\scshape Flynn, E.V.}
\newblock The group law on the {J}acobian of a curve of genus {$2$}.
\newblock {\em J. Reine Angew. Math.}, {\bf 439}:45--69 (1993). \mrev{1219694} (95b:14022), \zbl{0765.14014}. 

\bibitem[FS97]{flynn_and_smart}
{\scshape Flynn, E.V.; Smart, N.P.}
\newblock Canonical heights on the {J}acobians of curves of genus {$2$} and the
  infinite descent.
\newblock {\em Acta Arith.}, {\bf 79(4)}:333--352 (1997). \mrev{1450916} (98f:11066), \zbl{0895.11026}. 

\bibitem[Hol12a]{holmes2010canonical}
{\scshape Holmes, D. }
\newblock {Computing N\'eron-Tate heights of points on hyperelliptic
  Jacobians}.
\newblock {\em Journal of Number Theory}, {\bf132(6)}:1295 -- 1305 (2012). \mrev{2899805}, \zbl{1239.14019}. 

\bibitem[Hol12b]{holmesPhDThesis}
{\scshape Holmes, D. }
\newblock {\em {N\'eron-Tate heights on the Jacobians of high-genus
  hyperelliptic curves}}.
\newblock PhD thesis, University of Warwick (2012).

\bibitem[Hol12c]{first_version_of_this}
{\scshape Holmes, D.}
\newblock{\em {An Arakelov-theoretic approach to na\"ive heights on hyperelliptic Jacobians (version 1)}}
\newblock  {http://arxiv.org/abs/1207.5948v1}

\bibitem[Hri83]{hriljac_thesis}
{\scshape Hriljac, P. }
\newblock The {N}{\'e}ron-tate height and intersection theory on arithmetic
  surfaces.
\newblock {\em PhD Thesis, Massachusetts Institute of Technology} (1983). \mrev{2941042}. 



\bibitem[Lan88]{lang1988introduction}
{\scshape Lang, S. }
\newblock {\em {Introduction to Arakelov theory}}.
\newblock Springer (1988). \mrev{0969124} (89m:11059), \zbl{0667.14001}. 

\bibitem[Lip78]{Lipman}
{\scshape Lipman, J. }
\newblock Desingularization of two-dimensional schemes. 
\newblock {\em Ann. Math. (2)}, {\bf107 (1)} 151--207 (1978). \mrev{0491722} (58 \#10924), \zbl{0349.14004}. 

\bibitem[Liu02]{liu2006algebraic}
{\scshape Liu, Q.}
\newblock {\em Algebraic geometry and arithmetic curves}, volume~6 of {\em
  Oxford Graduate Texts in Mathematics}.
\newblock Oxford University Press, Oxford (2002).
\newblock Translated from the French by Reinie Ern{\'e}, Oxford Science
  Publications. \mrev{1917232} (2003g:14001), \zbl{0996.14005}. 

\bibitem[Man71]{manin_cyclo}
{\scshape Manin, Ju.~I.}
\newblock Cyclotomic fields and modular curves.
\newblock {\em Uspehi Mat. Nauk}, {\bf 26(6(162))}:7--71 (1971). \mrev{0401653} (53 \#5480), \zbl{0266.14012}. 

\bibitem[Mat80]{matsumura_commutative}
{\scshape Matsumura, H. }
\newblock {\em Commutative algebra}, {\bf 56} of {\em Mathematics Lecture Note
  Series}.
\newblock Benjamin/Cummings Publishing Co., Inc., Reading, Mass., second
  edition (1980). \mrev{0575344} (82i:13003), \zbl{0603.13001}. 

\bibitem[Moo20]{moore1920}
{\scshape Moore, E.H. }
\newblock  On the reciprocal of the general algebraic matrix. 
\newblock In {\em The fourteenth western meeting of the American Mathematical Society}, {\em Bull. Amer. Math. Soc.}, {\bf26}:385--396 (1920).  \mrev{1560324}. 


\bibitem[MB85]{Moret-Bailly1985Metriques-permi}
{\scshape Moret-Bailly, L.}
\newblock M{\'e}triques permises.
\newblock {\em Ast{\'e}risque}, {\bf127}:29--87 (1985). \mrev{0801918}, \zbl{1182.11028}. 

\bibitem[Mue10]{mullerthesis}
{\scshape Mueller, J.S.}
\newblock {Canonical heights on Jacobians}.
\newblock (2010).
\newblock Universit{\"a}t Bayreuth PhD thesis. 

\bibitem[{{Mue}}13]{SteffenComputation}
{\scshape Mueller, J.S.}
\newblock {Computing canonical heights using arithmetic intersection theory}.
\newblock {\em Math. Comp.}, DOI: http://dx.doi.org/10.1090/S0025-5718-2013-02719-6,  (2013). \mrev{3120591}, \zbl{06227557}. 

\bibitem[Mum66]{mumford_equations_1}
{\scshape Mumford, D.}
\newblock On the equations defining abelian varieties. {I}.
\newblock {\em Invent. Math.}, {\bf 1}:287--354 (1966). \mrev{0204427}, ({34 \#4269)} \zbl{0219.14024}. 

\bibitem[Mum84]{mumfordTataII}
{\scshape Mumford, D.}
\newblock {\em {Tata lectures on theta II}}.
\newblock Birkh{\"a}user (1984). \mrev{0742776} {(86b:14017)}, \zbl{1112.14003}. 


\bibitem[N{\'e}r65]{neron1965quasi}
{\scshape N{\'e}ron, A.}
\newblock {Quasi-fonctions et hauteurs sur les vari{\'e}t{\'e}s
  ab{\'e}liennes}.
\newblock {\em Annals of Mathematics}, {\bf82(2)}:249--331 (1965). \mrev{0179173} {(31 \#3424)}, \zbl{0163.15205}. 

\bibitem[Pen55]{MR0069793}
{\scshape Penrose, R.}
\newblock A generalized inverse for matrices.
\newblock {\em Proc. Cambridge Philos. Soc.}, {\bf51}:406--413 (1955). \mrev{0069793}{ (16,1082a)}, \zbl{0065.24603}. 

\bibitem[Rei72]{ReidThesis}
{\scshape Reid, M.}
\newblock {\em The complete intersection of two or more quadrics}.
\newblock PhD thesis, University of Cambridge, (1972).

\bibitem[Sik95]{siksek1995infinite}
{\scshape Siksek, S.}
\newblock Infinite descent on elliptic curves.
\newblock {\em Rocky Mountain Journal of Mathematics}, {\bf25(4)} (1995). \mrev{1371352}{ (97g:11053)}, \zbl{0852.11028}. 

\bibitem[Sil90]{MR1035944}
{\scshape Silverman, J.H.}
\newblock The difference between the {W}eil height and the canonical height on
  elliptic curves.
\newblock {\em Math. Comp.}, {\bf 55(192)}:723--743 (1990). \mrev{1035944}{ (91d:11063)}, \zbl{0729.14026}. 

\bibitem[Sil09]{Silverman2009The-arithmetic-}
{\scshape Silverman, J.H.}
\newblock {\em The arithmetic of elliptic curves}, volume {\bf106} of {\em Graduate
  Texts in Mathematics}.
\newblock Springer, Dordrecht, second edition, (2009). \mrev{1329092}{ (95m:11054)}, \zbl{1194.11005}. 

\bibitem[Sto99]{stoll1999height}
{\scshape Stoll, M.}
\newblock On the height constant for curves of genus two.
\newblock {\em Acta Arith}, {\bf90(2)}:183--201 (1999). \mrev{1709054}{ (2000h:11069)}, \zbl{0932.11043}. 

\bibitem[Sto02]{stoll2002height}
{\scshape Stoll, M.}
\newblock On the height constant for curves of genus two, ii.
\newblock {\em Acta Arith}, {\bf104(2)}:165--182 (2002). \mrev{1914251}{ (2003f:11093)}, \zbl{1139.11318}. 

\bibitem[Sto12]{stoll_slides}
{\scshape Stoll, M.}
\newblock Explicit kummer varieties for hyperelliptic curves of genus 3. {\em Slides from a talk}
\newblock (2012).
%

\bibitem[Uch08]{uchida2008difference}
{\scshape Uchida, Y.}
\newblock The difference between the ordinary height and the canonical
  height on elliptic curves.
\newblock {\em J. Number Theory} {\bf128}, no. 2, 263-279 (2008).  \mrev{2380321}{ (2009f:11078)}, \zbl{1145.11050}. 

\bibitem[VW98]{van1998equations}
{\scshape van~Wamelen,~P.}
\newblock Equations for the jacobian of a hyperelliptic curve.
\newblock {\em Trans. Amer. Math. Soc.}, {\bf 350}:3083--3106 (1998). \mrev{1432144 }{(98k:14038)}, \zbl{0901.14016}. 

\bibitem[Zim76]{MR0419455}
{\scshape Zimmer, G.H.}
\newblock On the difference of the {W}eil height and the {N}\'eron-{T}ate
  height.
\newblock {\em Math. Z.}, {\bf 147(1)}:35--51 (1976). \mrev{0419455}{ (54 \#7476)}. 

\bibitem[ZM72]{zarhin_manin_heights}
{\scshape  Zarhin, Ju.~G.; Manin, Ju.~I. }
\newblock Height on families of abelian varieties.
\newblock {\em Mat. Sb. (N.S.)}, {\bf 89(131)}:171--181, 349 (1972). \mrev{0332801}{ (48 \#11127)}. 

\end{thebibliography}

\newcommand{\mrev}[1]{MR#1}
\newcommand{\zbl}[1]{Zbl #1}

\newcommand{\etalchar}[1]{$^{#1}$}
\newcommand{\cprime}{$'$}

\end{document}